\documentclass[a4paper]{article}
\usepackage{geometry}
\usepackage{graphicx,url}
\usepackage[margin=10pt, font=small, labelfont=bf]{caption}

\usepackage{slashed}
\usepackage{mathtools}
\mathtoolsset{showonlyrefs}

\usepackage[notref]{showkeys}

\usepackage[english]{babel}   
\usepackage[T1]{fontenc}
\usepackage[utf8]{inputenc}  

\usepackage{tikz}
\usepackage{verbatim}
\usepackage{listings}
\usepackage{xcolor}
\usepackage{amsmath,amssymb,mathrsfs}
\usepackage{mathtools}
\usepackage{amsthm}
\usepackage{frontespizio}
\usepackage{pictexwd, dcpic}
\usepackage{epigraph}
\usepackage{enumerate}
\usepackage{fancyhdr}
\usepackage{textcomp}
\usepackage{cancel}
\usepackage{bbm}
\theoremstyle{plain}
\newtheorem{thm}{Theorem}[section]

\newtheorem{cor}[thm]{Corollary}
\newtheorem{lem}[thm]{Lemma}
\theoremstyle{definition}
\newtheorem{defn}[thm]{Definition}
\newtheorem{rem} [thm] {Remark}

\title{Local and global existence for the Ericksen - Leslie problem in unbounded domains}
\date{}
\author{Daniele Barbera and Vladimir Georgiev}
\begin{document}
\maketitle

\begin{abstract}
    The work deals with the Ericksen-Leslie System for nematic liquid crystals on the space $\mathbb{R}^N$ with $N\ge 3$, on $\mathbb R^3_+$ and on $\Omega\subseteq\mathbb R^3$ exterior domain with sufficiently smooth boundary. The crystal orientation is described by unit vector $v$ that is a small perturbation of a fixed constant vector $\eta$. We prove through a combination of  energy method with dispersive  a priori estimates a local existence and global existence for small initial data by a contraction arguments. In particular, we obtain the following regularity of the liquid velocity $u$ and of the crystal orientation $v$  
    $$ u\in L^\infty((0,T);H^s(\Omega)),\quad \nabla u\in L^2((0,T);H^s(\Omega)), $$
    $$ \nabla v\in L^\infty((0,T);H^s(\Omega)),\quad \nabla^2v\in L^2((0,T);H^s(\Omega)) $$
    for $s>\frac{N}{2}-1$ if $\Omega=\mathbb R^N$ and $s\in\left(\frac{1}{2},1\right]$ if $\Omega=\mathbb R^3_+$ or in the exterior case, asking low regularity assumptions  on $u_0$ and $v_0$.

\end{abstract}
{Key words: Liquid crystals, Ericksen - Leslie, heat equation, Stokes equation, energy estimates}\\
{AMS Subject Classification 2010: 76A15, 35Q30, 35Q35}

\section{Introduction}
Let $N\ge 3$ and $\Omega\subseteq \mathbb R^N$ and $T>0$, then we consider the Ericksen-Leslie system
\begin{equation}\label{EL.sys.}
    \left\{\begin{array}{ll}
       (\partial_t-\Delta)u+\nabla p+u\cdot \nabla u=-{\rm Div}\left(\nabla v\odot\nabla v\right)  & (0,T)\times \Omega \\
       {\rm div}u=0 & (0,T)\times\Omega \\
       (\partial_t-\Delta)v +u\cdot \nabla v = |\nabla v|^2v & (0,T)\times\Omega \\
       |v|=1 & (0,T)\times\Omega \\
       \partial_\nu v=0,\quad u=0 & (0,T)\times \partial\Omega \\
       u(0)=u_0,\quad v(0)=v_0 & \Omega,
    \end{array}\right.     
\end{equation}
where $u,v\colon (0,T)\times \Omega\to \mathbb R^N$, $p\colon(0,T)\times\Omega\to \mathbb R$, $\nu(x)$ is the external normal vector in $x\in\partial\Omega$ and 
$$ {\rm Div}A(x)=\sum_{j=1}^N\partial_jA^j(x)\quad \forall A\colon\Omega\mapsto \mathbb R^{N^2}, $$
$$ [\nabla z\odot \nabla w]_{i,j}=\partial_iz\cdot \partial_jw\quad \forall z,w\colon \Omega\to\mathbb R^N,\quad i,j=1,\ldots, N. $$
In the paper, we consider $\Omega=\mathbb R^N,\mathbb R^N_+$ or $\Omega\subseteq\mathbb R^N$ exterior domain with sufficiently smooth boundary. 

\vspace{2mm}

The Ericksen-Leslie model was introduced in  \cite{E61} and \cite{L68} and it describes by macroscopic theory and continuum mechanics the behaviour of nematic liquid crystals. Liquid crystals are a state of matter intermediate between the liquid and the solid state. More precisely, the nematic liquid ones described in Ericksen-Leslie Model \eqref{EL.sys.} are formed by rod-like molecules with no positional order. The uniaxial structure of the particles makes the material anisotropic, like crystals. On the other hand, liquid crystals flow similarly to liquids. For this reason, the system is formed by a Navier-Stokes equation for the velocity field $u$ and the pressure $p$ coupled with a heat equation for the average direction field $v$. The rigorous derivation of the model uses the conservation laws for mass, momentum and angular momentum as well on constitutive relations given by Leslie \cite{L68}. Some further discussions on the model can be found in  \cite{Ch92}, \cite{S04}, \cite{Z21}.

\subsection{Known results}

As in the case of Navier - Stokes problem, variety of results can be divided in the following groups:

\begin{enumerate}
    \item[a)] Existence of weak solutions in $\mathbb{R}^N$ or domains;
    \item[b)] Existence and uniqueness  of local strong solution in $\mathbb{R}^N$ that can be extended to a global one provided initial data are small;
    \item[c)] Existence and uniqueness of local strong solution in domains in $\mathbb{R}^N$ that can be extended to a global one provided initial data are small.
\end{enumerate}

Without pretending to have complete description of all results on Ericksen-Leslie  model, we shall mention some of the works close to our main goal to study the point c) listed above.

The existence of weak solutions with initial data $u_0 \in L^2(\Omega),$  $d_0 \in H^1(\Omega)$
was initiated 
 in \cite{L89} and \cite{LL95} in the case of open domains in $\mathbb{R}^N$, $N=2,3.$  The existence of unique classical solution with initial data $u_0 \in H^1(\Omega),$  $d_0 \in H^2(\Omega)$ is also  established in \cite{L89}.
 The weak solutions in the whole space are studied in 
\cite{DFRSS16}.

The existence of local in time strong solution in $\mathbb{R}^N$  can be extended to a global one provided initial data are small (see \cite{Li18}, \cite{LZ16} for example).
More precisely, in these works the case of $\mathbb{R}^3$ is considered and  the initial data are taken  in $\dot{B}^{3/p-1}_{p.1}$. Note that for $p=2$ we have 
initial data in $\dot{B}^{1/2}_{2.1}$ that is a slightly smaller than classical Sobolev space $\dot{H}^{1/2}$. Local and Global existence for strong solution in $\mathbb R^2$ or $\mathbb R^3$ for $u_0\in H^s$ and $d_0\in H^{s+1}$ with $s>\frac{N}{2}-1$ can be found in \cite{JLT19} and \cite{GJ22}. We proved in \cite{BG23} the existence of local and global solutions for \eqref{EL.sys.} in $H^s$ with $s>\frac{N}{2}-1$ and $N\ge 3$ with a geometrical assumption on $v$.

The case of a bounded $C^3$ domain in $\mathbb{R}^3$ is studied  \cite{HW10} when the initial data belong to a Besov class spaces $B^{s}_{q.p}(\Omega)$
with $s = 2(1-1/p)$ and $1 \leq p,q \leq \infty$ satisfying $$ \frac{2}{p}\left( 1- \frac{3}{q}\right) \in (0,1).$$
Note that for $p=1$ the condition $q>3$ has to be required.

The maximal $L^p$ regularity estimates for abstract evolution semigroups have been applied in
\cite{HP19} for the case of non-isothermal Ericksen - Leslie mode.

Since the equation for $u$ in \eqref{EL.sys.} is similar to Navier - Stokes equation, in the case of bounded $C^3$ domains in $\mathbb{R}^3$ we can recall the classical result of  Fujita and Kato in \cite{FK64} , where they proved the well - posedness for the case of initial data $u_0 \in H^{1/2}.$

As Fujita and Kato have predicted in their work, it is natural to start with $L^2$ theory and then to move to $L^p$ setting.In fact, this is already well - done for Navier - Stokes case and  we refer  for example to  \cite{FGH19} and references there.
More precisely, the results in \cite{FGH19} allow  one to take initial data in $\dot{B}^{3/p-1}_{p, s}$ in case of bounded $C^{2,1}$ domains in $\mathbb{R}^3$  and see that  the results cover the case c) listed above (for the $L^p$ theory for the heat case see \cite{BG22}).

\subsection{Solenoidal Space in $L^2(\Omega)$}\label{subsec.sol.spaces}

Let $q\in(1,\infty)$. From \cite{So01} we recall the definition
$$ G_q(\Omega)=\left\{w\in L^q(\Omega)\mid w=\nabla \rho \:\:\text{for some}\:\:\rho\in L^1_{{\rm loc}}(\Omega)\:\:\text{with}\:\:\nabla p\in L^q(\Omega)\right\}, $$
$$ J_q(\Omega)=\overline{C^\infty_\sigma(\Omega)}^{\|\cdot\|_{L^q(\Omega)}}, $$
where 
\begin{equation}\label{def.Cinf-sigma}
    C^\infty_\sigma(\Omega)=\left\{f\in C^\infty_c(\Omega)\mid {\rm div}f=0\right\}.
\end{equation}
For our choice of $\Omega$ (see Theorem 2 and 3 at p.128 and 129 of \cite{GN18}) we have the Helmholtz Decomposition:
$$ L^q(\Omega)=J_q(\Omega)\oplus G_q(\Omega), $$
where $J_q(\Omega)$ and $G_q(\Omega)$ are orthogonal. In particular we denote $ \mathbb P_q\colon L^q(\Omega;\mathbb R^N)\to J_q(\Omega)$ the projection on $J_q(\Omega)$, which is called Helmholtz Projection and we will denote $\mathbb P=\mathbb P_2$. To be noticed that, from \cite{GN18}, $\mathbb P_p=\mathbb P_q$ in $L^p(\Omega)\cap L^q(\Omega)$ for any $p,q\in(1,\infty)$ and $\Omega\subseteq \mathbb R^N$ as in our case. For simplicity, in the following we write $\mathbb P$ without specify the index $q$.
\begin{rem}\label{rem.der.P.}
    Let $\Omega=\mathbb R^N,\mathbb R^N_+$ or an exterior domain with sufficiently smooth boundary, then $\mathbb P\colon H^1(\Omega)\to H^1(\Omega)$. Moreover, if $\Omega=\mathbb R^N$, then
    $$ \mathbb P \nabla u =\nabla \mathbb P u \quad \forall u\in H^1\left(\mathbb R^N\right). $$
    For what concerns the property of $\mathbb P\colon H^1(\Omega)\to H^1(\Omega)$, we have the decomposition
    $$ u=\mathbb P u+\nabla \pi\quad \nabla \pi\in G_2(\Omega). $$
    Thanks to Lemma III.1.2 of \cite{G11} the Helmholtz decomposition is equivalent to the weak solvability of the system
    $$ \left\{\begin{array}{ll}
       \Delta \pi={\rm div}u  & \Omega \\
       \partial_\nu \pi= \nu\cdot u  & \partial\Omega.
    \end{array}\right. $$
    Since $u\in H^1(\Omega)$, by elliptic estimates, we get appropriate apriori bounds. In particular $\nabla \pi\in H^1(\Omega)$ and therefore $\mathbb Pu\in H^1(\Omega)$.     Finally, let $\Omega=\mathbb R^N$, then for any $j=1,\ldots, N$
    $$ \partial_ju=\partial_j\mathbb Pu + \nabla \partial_j\pi. $$
    We know that $\nabla \partial_j\pi\in G_2(\mathbb R^N)$ and ${\rm div}\partial_j\mathbb Pu=0$, so 
    $$ \mathbb P\partial_j u = \partial_j\mathbb P u. $$
\end{rem}

We introduced the Helmholtz projection because we will work with low regularity initial data in classical Sobolev spaces and, for domains different from $\mathbb R^N$, it is not easy to deal with the pressure term directly. For this reason we consider the system after applying the Helmholtz Projection:
\begin{equation}\label{EL.sys-proj.}
    \left\{\begin{array}{ll}
       (\partial_t-\mathbb P\Delta)u =-\mathbb P\left[u\cdot \nabla u + {\rm Div}\left(\nabla v\odot\nabla v\right)\right]  & (0,T)\times \Omega \\
       {\rm div}u=0 & (0,T)\times\Omega \\
       (\partial_t-\Delta)v +u\cdot \nabla v = |\nabla v|^2v & (0,T)\times\Omega \\
       |v|=1 & (0,T)\times\Omega \\
       u=0, \quad \partial_\nu v=0 & (0,T)\times \partial\Omega \\
       u(0)=u_0,\quad v(0)=v_0 & \Omega.
    \end{array}\right.     
\end{equation}
We conclude this part with the following result for the pressure term from Lemma 1.4.2 at p.202 of \cite{So01}:
\begin{lem}\label{l.ex.press.}
Let $N\ge 2$, $\Omega\subseteq\mathbb R^N$ be a general domain, let $\Omega_0\subseteq\Omega$ with $\Omega_0\not=\emptyset$ bounded with $\overline{\Omega_0}\subseteq\Omega$, let $T\in(0,\infty]$, $p,q\in(1,\infty)$, let $f\in L^p((0,T);W^{-1,q}_{loc}(\Omega))$ such that
$$ \left<f,v\right>=0 \quad \forall v\in C^\infty_0((0,T);C^\infty_\sigma(\Omega)), $$
where $C^\infty_\sigma(\Omega)$ is defined in \eqref{def.Cinf-sigma} and $\left<,\right>$ is the duality product in $W^{-1,q}-W^{1,q^\prime}$, then there is a unique $\pi\in L^p((0,T);L^q_{loc}(\Omega))$ such that 
$$ \nabla \pi=f,\quad \int_{\Omega_0}\pi(t)dx=0\quad \text{for a.e.}\:\:t\in(0,T). $$
Moreover, for any $\Omega^\prime\subseteq\Omega$ bounded Lipshitz domain such that $\Omega_0\subseteq\Omega^\prime$ and $\overline{\Omega^\prime}\subseteq\Omega$, we have
$$ \|\pi\|_{L^p((0,T);L^q(\Omega^\prime))}\le C(q,\Omega,\Omega^\prime)\|f\|_{L^p((0,T);W^{-1,q}(\Omega^\prime))}. $$
\end{lem}
We recall in the previous result that for any $U\subseteq\mathbb R^N$ open set
$$ W^{-1,q}(U)=\left(W^{1,q^\prime}_0(U)\right)^\prime, $$
where
$$ W^{1,q^\prime}_0(U)=\overline{C^\infty_c(U)}^{\|\cdot \|_{W^{1,q^\prime}(U)}}. $$

\subsection{Main results}

The aim of the paper is to prove local and global existence for the Ericksen-Leslie problem \eqref{EL.sys.} and \eqref{EL.sys-proj.} in $\Omega=\mathbb R^N,\mathbb R^N_+$ or $\Omega\subseteq\mathbb R^N$ exterior domain with sufficiently smooth boundary, generalizing what we did in \cite{BG23}. We start from the following assumption:
\begin{equation}\label{ass.eta}
    v(t,x)=\eta+d(t,x)\quad \text{for a.e.}\:\:(t,x)\in(0,T)\times\Omega,
\end{equation} 
with $\eta\in\mathbb R^N$ fixed. Given \eqref{ass.eta} we can rewrite the problem \eqref{EL.sys-proj.} as a system for the functions $(u,d)$:
\begin{equation}\label{EL.sys.eta-ass.}
    \left\{\begin{array}{ll}
       (\partial_t-\mathbb P\Delta)u= - \mathbb P\left(u\cdot \nabla u + {\rm Div}\left(\nabla d\odot\nabla d\right)\right)  & \mathbb{R}_+\times \Omega \\
       {\rm div}u=0 & \mathbb{R}_+\times\Omega \\
       (\partial_t-\Delta)d +u\cdot \nabla d = |\nabla d|^2(\eta+d) & \mathbb{R}_+\times\Omega \\
       |\eta+d|=1 & \mathbb{R}_+\times\Omega \\
       u=0, \quad \partial_\nu d=0 & \mathbb R_+\times\partial\Omega \\
       u(0)=u_0,\quad d(0)=d_0 & \Omega. 
    \end{array}\right.     
\end{equation}
We are almost ready to state the main results. Before, we need to introduce some spaces: 
\begin{equation}\label{def.X}
    X^s_T(\Omega)\coloneqq L^\infty\left((0,T);H^s(\Omega)\right) \cap L^2\left((0,T);H^{s+1}(\Omega)\right),
 \end{equation}
for $\Omega\subseteq\mathbb R^N$ with $N\ge 3$, $s\ge 0$ and $T\in(0,\infty]$. For $T=\infty$ we write $X^s(\Omega)$. Our strategy is to apply a contraction argument on these spaces. The initial conditions will belong to the following spaces:
\begin{defn}\label{def.HA}\hfill\\
Let $N\ge 3$, $\Omega\subseteq\mathbb R^N$, $s\ge 0$, let $A\colon D(A)\subseteq L^2(\Omega)\to L^2(\Omega)$ be a negative, self-adjoint operator, then we denote with $H^s_A(\Omega)$ the space $D((-A)^{s/2})$ endowed with the norm
$$ \|f\|_{H^s_A(\Omega)}\coloneqq \left\|(1-A)^{s/2}f\right\|_{L^2(\Omega)}. $$
\end{defn}
In particular, we consider $A=\Delta_D,\Delta_N,\mathbb P\Delta_D$, which are respectively the Laplace operator with Dirichlet conditions, the Laplace operator with Neumann conditions and the Stokes operator with the Dirichlet conditions.

\vspace{2mm}

As we mentioned above, we consider $\Omega=\mathbb R^N,\mathbb R^N_+$ and $\Omega\subseteq \mathbb R^N$ exterior domain with sufficiently smooth boundary. Let us start with $\Omega=\mathbb R^N$:
\begin{thm}\label{t.loc.ex.RN}
    Let $N\ge 3$, $s>\frac{N}{2}-1$, let $\eta\in\mathbb R^N$, $u_0\colon\mathbb R^N\to \mathbb R^N$ and $v_0\colon\mathbb R^N\to S^{N-1}$ with
    $$ u_0\in H^s\left(\mathbb R^N;\mathbb R^N\right),\quad v_0=\eta + d_0,\quad \text{where}\:\: d_0\in H^{s+1}\left(\mathbb R^N;\mathbb R^N\right), $$
    then, for any $R>0$ such that
    $$ \|u_0\|_{H^s(\mathbb R^N)}+\|d_0\|_{H^{s+1}(\mathbb R^N)} \le R, $$
    we can find $T_0=T_0(R)$ such that, for any $T\in(0,T_0)$ there is a unique solution $(u,p,v)$, up to additive functions $\rho(t)$ in the pressure term, for \eqref{EL.sys.} in $(0,T)$ such that
    $$ u\in X^s_T\left(\mathbb R^N\right),\quad v\in X^{s+1}_T\left(\mathbb R^N\right), \quad p\in L^2\left((0,T);L^2_{loc}\left(\mathbb R^N\right)\right). $$
\end{thm}
\begin{thm}\label{t.gl.ex.RN}
    Let $N\ge 3$, $s>\frac{N}{2}-1$, let $\eta\in\mathbb R^N$, $u_0\colon\mathbb R^N\to\mathbb R^N$ and $v_0\colon\mathbb R^N\to S^{N-1}$ with
    $$ u_0\in H^s\left(\mathbb R^N;\mathbb R^N\right)\cap L^1\left(\mathbb R^N;\mathbb R^N\right),\quad v_0=\eta + d_0, \quad \text{where}\:\:d_0 \in H^{s+1}\left(\mathbb R^N;\mathbb R^N\right)\cap L^1\left(\mathbb R^N;\mathbb R^N\right), $$
    then we can find $\varepsilon_0>0$ such that, for any $\varepsilon\in(0,\varepsilon_0)$ with
    $$ \|u_0\|_{H^s\cap L^1(\mathbb R^N)} + \| v_0-\eta\|_{H^{s+1}\cap L^1(\mathbb R^N)}\le \varepsilon, $$
    we can find a solution $(u,p,v)$, unique up to additive functions $\rho(t)$ in the pressure term, for \eqref{EL.sys.} such that
    $$ u\in X^s\left(\mathbb R^N\right),\quad v\in X^{s+1}\left(\mathbb R^N\right), \quad p\in L^2\left(\mathbb R_+;L^2_{loc}\left(\mathbb R^N\right)\right). $$
    Moreover, for any $k\in\mathbb N$ such that $s-k>\frac{N}{2}-1$, it holds
    $$ \|\nabla^j u(t)\|_{L^\infty(\mathbb R^N)}\le C t^{-\frac{N}{2}-\frac{j}{2}}\quad t\ge 1, \quad j=0,\ldots, k, $$
    $$ \|\nabla^j(v(t)-\eta)\|_{L^\infty(\mathbb R^N)}\le Ct^{-\frac{N}{2}-\frac{k}{2}}\quad t\ge 1,\quad j=0,\ldots, k+1. $$
\end{thm}
With respect with the works \cite{JLT19} and \cite{GJ22}, these results generalize the dimension $N\ge 3$ and prove a decay in time of the solutions in $\mathbb R_+$. 

Let us state now the results for the half-plane and the exterior cases:
\begin{thm}\label{t.loc.ex.dom.}
    Let $s\in\left(\frac{1}{2},1\right]$, let $\Omega=\mathbb R^3_+$ or $\Omega\subseteq \mathbb R^3$ be an exterior domain with sufficiently smooth boundary, let $\eta\in\mathbb R^3$, $u_0\colon\Omega\to\mathbb R^3$ and $v_0\colon\Omega\to S^2$ with
    $$ u_0\in H^s_{\mathbb P\Delta_D}\left(\Omega;\mathbb R^3\right),\quad v_0=\eta + d_0,\quad \text{where}\:\:d_0\in H^{s+1}_{\Delta_N}\left(\Omega;\mathbb R^{3}\right), $$
    then, for any $R>0$ such that
    $$ \|u_0\|_{H^s(\Omega)}+\|v_0-\eta\|_{H^{s+1}(\Omega)} \le R, $$
    we can find $T_0=T_0(R)$ such that, for any $T\in(0,T_0)$ there is a solution $(u,p,v)$, unique up to additive functions $\rho (t)$ in the pressure term, for \eqref{EL.sys.} in $(0,T)$ with
    $$ u\in X^s_T(\Omega),\quad v\in X^{s+1}_T(\Omega),\quad p\in L^2\left(L^2(0,T);L^2_{loc}(\Omega)\right). $$
\end{thm}
\begin{thm}\label{t.gl.ex.dom.}
    Let $s\in\left(\frac{1}{2},1\right]$, let $\Omega=\mathbb R^3_+$ or $\Omega\subseteq \mathbb R^3$ be an exterior domain with sufficiently smooth boundary, let $\eta\in\mathbb R^3$, $u_0\colon\Omega\to\mathbb R^3$ and $v_0\colon\Omega\to S^2$ with
    $$ u_0\in H^s_{\mathbb P\Delta_D}\left(\Omega;\mathbb R^3\right)\cap L^1\left(\Omega;\mathbb R^3\right),\quad v_0=\eta + d_0, \quad \text{where}\:\:d_0  \in H^{s+1}_{\Delta_N}\left(\Omega;\mathbb R^3\right)\cap L^1\left(\Omega;\mathbb R^3\right), $$
    then we can find $\varepsilon_0>0$ such that, for any $\varepsilon\in(0,\varepsilon_0)$ with
    $$ \|u_0\|_{H^s_{\mathbb P\Delta_D}\cap L^1(\Omega)} + \| v_0-\eta\|_{H^{s+1}_{\Delta_N}\cap L^1\cap L^\infty(\Omega)}\le \varepsilon, $$
    we can find a unique solution $(u,p,v)$, unique up to additive functions $\rho (t)$ in the pressure term, for \eqref{EL.sys.} in $\mathbb R_+$ with
    $$ u\in X^s(\Omega),\quad v\in X^{s+1}(\Omega),\quad p\in L^2\left(\mathbb R_+;L^2_{loc}(\Omega)\right). $$
    Moreover, when $\Omega=\mathbb R^3_+$, it holds
    $$ \|u(t)\|_{L^\infty(\mathbb R^3_+)}\le C t^{-\frac{3}{2}}\quad t\ge 1, $$
    $$ \|\nabla^j(v(t)-\eta)\|_{L^\infty(\mathbb R^3_+)}\le C t^{-\frac{3}{2}-\frac{j}{2}}\quad t\ge 1,\quad j=0,1, $$
    while for $\Omega$ exterior domain with sufficiently smooth boundary it holds
    $$ \|u(t)\|_{L^\infty(\mathbb R^3_+)}\le C t^{-\frac{1}{2}}\quad t\ge 1, $$
    $$ \|\nabla^j(v(t)-\eta)\|_{L^\infty(\mathbb R^3_+)}\le C t^{-\frac{3}{2}-\frac{j}{2}}\quad t\ge 1,\quad j=0,1. $$
\end{thm}

We notice that, in the half-space and in the exterior case, we only consider $N=3$ and $s\in\left(\frac{1}{2},1\right]$. In fact, the estimates for the nonlinearities in $X^s_T(\Omega)$ are done for $s>\frac{N}{2}-1$ for any $\Omega$ we considered. On the other hand, in $\mathbb R^N$ the linear estimates are true for any $s\ge 0$, while the same result is true for the half-space and the exterior domains only for $s\in [0,1]$. That is way we need to ask $N=3$ and $s\in\left(\frac{1}{2},1\right]$.

\section{Preliminary Results}\label{sec.prel.res.}

\subsection{Sobolev Spaces and Domains of Operators}

In this paper we will use some functional spaces and the aim of this subsection is to describe them and list some properties we use. Everything can be found in \cite{T83}. We recall the well-known property:
$$ H^s(\Omega)=W^{s,2}(\Omega)\quad s\ge 0, $$
where 
$$ W^{s,p}(\Omega)=\left\{f\in L^p(\Omega)\mid \exists F\in W^{s,p}\left(\mathbb R^N\right),\:\:F_{|\Omega}=f\right\}, $$
$$ H^s_p(\Omega)=\left\{f\in L^p(\Omega)\mid \exists F\in H^s_p\left(\mathbb R^N\right),\:\:F_{|\Omega}=f\right\}, $$
and
$$ \|f\|_{W^{m,p}(\mathbb R^N)}=\sum_{|\alpha|\le m}\|D^\alpha f\|_{L^p(\mathbb R^N)}\quad m\in\mathbb N, $$
$$ \|f\|_{W^{s,p}(\mathbb R^N)}= $$
$$ = \|f\|_{W^{\lfloor s\rfloor,p}(\mathbb R^N)} + \sum_{|\alpha|=s}\left(\iint_{\mathbb R^N\times\mathbb R^N}\frac{|D^\alpha f(x+h)-D^\alpha f(x)|^p}{|h|^{N +(s-\lfloor s \rfloor )p}}dxdh\right)^\frac{1}{p}<+\infty\} \quad s\not\in\mathbb N, $$
$$ \|f\|_{H^s_p(\mathbb R^N)}= \left\|\mathcal{F}^{-1}\left[(1+|\xi|^2)^{s/2}\mathcal{F}[f]\right]\right\|_{L^p(\mathbb R^N)}.  $$
In particular
\begin{equation}\label{int.res.}
        \left[H^{s_0}(\Omega),H^{s_1}(\Omega)\right]_\theta=H^s(\Omega),
    \end{equation}
\begin{equation}\label{Hs-norm.}
    \|f\|_{H^s(\Omega)}=\|f\|_{W^{s,2}(\Omega)}=\|f\|_{H^{\lfloor s\rfloor}(\Omega)} + \left(\iint_{\Omega\times\Omega} \frac{|f(x)-f(y)|^2}{|x-y|^{N+(s-\lfloor s\rfloor )2}}dxdy\right)^\frac{1}{2}.
\end{equation} 
Finally, we want to prove an equivalence between the $W^{s,2}(\Omega)$-norm and the $H^s_A(\Omega)$-norm from Definition \ref{def.HA} for $A=\Delta_D,\Delta_N,\mathbb P\Delta_D$. By Spectral Theorem, it can be seen that 
\begin{thm}\label{t.int.res.2}\hfill\\
Let $N\ge 3$ and $\Omega\subseteq\mathbb R^N$ open with sufficiently smooth boundary, let $\theta\in(0,1)$ then
$$ \left[L^2(\Omega),H^2_A(\Omega)\right]_\theta=H^{2\theta}_A(\Omega). $$
\end{thm}
We are finally ready to prove resolvent equivalences for $A=\Delta_D,\Delta_N,\mathbb P\Delta_D$:
\begin{lem}\label{l.res.eq.}
  Let $N\ge 3$, $\Omega=\mathbb R^N,\mathbb R^N_+$ or $\Omega\subseteq\mathbb R^N$ exterior domain with sufficiently smooth boundary, let $s\in[0,2]$, then it holds
  $$ \|(1-\Delta_N)^{s/2}f\|_{L^2(\Omega)} \sim \|f\|_{H^s(\Omega)}, $$
  $$ \|(1-\Delta_D)^{s/2}f\|_{L^2(\Omega)}\sim \|f\|_{H^s(\Omega)}, $$
  $$ \|(1-\mathbb P\Delta_D)^{s/2}f\|_{L^2(\Omega)} \sim \|f\|_{H^s(\Omega)}. $$
\end{lem}
\begin{proof}\hfill\\
    The case $s=0$ is obvious. For what concerns the case $s=2$, the inequality
    $$ \|(1-A)f\|_{L^2(\Omega)}\le \|f\|_{H^2(\Omega)}, $$
    is again clear.  The inverse estimate follows from resolvent estimates, well-known for $A=\Delta_D,\Delta_N$ and a consequence of Theorem 3.1 of \cite{SS12}, Theorem 1.3 of \cite{SS12} and Theorem 1.2 of \cite{FS94} for $A=\mathbb P\Delta_D$. Then, thanks to the interpolation result \eqref{int.res.} and Theorem \ref{t.int.res.2} we conclude. 
\end{proof}

\subsection{Semigroup Estimates}\label{subsec.sem.est.}

In this part, we want to list some tools regarding the semigroup associated with the operators $A=\Delta_D,\Delta_N,\mathbb P\Delta_D$. They will be useful in the linear estimates and in the last part of the paper, devoted to the time decay of the solution. Let us start with the $L^p-L^q$ estimates:
\begin{thm}\label{t.sem.es.0}\hfill\\
Let $\Omega=\mathbb R^N,\mathbb R^N_+$ and $\Omega\subseteq\mathbb R^N$ exterior domain with sufficiently smooth boundary for $N\ge 3$, let $1\le p \le q\le \infty$, let $A=\Delta_D,\Delta_N$, then for any $f\in L^p(\Omega)$ it holds
$$ \left\|e^{At}f\right\|_{L^q(\Omega)}\le C(\Omega,p,q) t^{-\frac{N}{2}\left(\frac{1}{p}-\frac{1}{q}\right)}\|f\|_{L^p(\Omega)}. $$
Moreover, if $1<p\le q<\infty$ or $q=\infty$ and $p\in(1,\infty)$, then for any $f\in L^p(\Omega)$ it holds 
$$ \left\|e^{\mathbb P\Delta_Dt}\mathbb Pf\right\|_{L^q(\Omega)}\le C(\Omega,p,q) t^{-\frac{N}{2}\left(\frac{1}{p}-\frac{1}{q}\right)}\|f\|_{L^p(\Omega)}. $$
\end{thm}
\begin{proof}\hfill\\
For what concerns the heat case, in $\mathbb R^N$ and $\mathbb R^N_+$ the estimate is clear because we know explicitly the Green function. The proof for $\Omega$ exterior domain follows by Proposition 3.1 of \cite{IMT18} and Section 2 and Theorem 1.1 of \cite{I09}. For the Stokes case, the estimate in $\mathbb R^N$ follows from the fact that 
\begin{equation}\label{St-Heat.sem.RN}
    e^{\mathbb P\Delta t}\mathbb Pf=e^{\Delta t}\mathbb Pf.
\end{equation}
The half-plane case can be verified using the estimates for the Green function in Proposition 1.1 of \cite{KLT23}. Finally, the exterior case follows from Theorem 1.2 of \cite{MS97}.
\end{proof}

We also need an estimate for the gradient of the semigroups. Differently from the previous estimates, here the results for the exterior case are not always the same of the $\mathbb R^N$ case. 

\begin{thm}\label{t.sem.es.1-easy}
    Let $k\in\mathbb N$, $1\le p\le q\le \infty$, let $\Omega=\mathbb R^N,\mathbb R^N_+$ with $N\ge 3$, let $A=\Delta_D,\Delta_N$, then for any $f\in L^p(\Omega)$ it holds
    $$ \left\|\nabla^k e^{At}f\right\|_{L^q(\Omega)}\le C(N,p,q,k) t^{-\frac{k}{2}-\frac{N}{2}\left(\frac{1}{p}-\frac{1}{q}\right)}\|f\|_{L^p(\Omega)}\quad t>0. $$
    Moreover, if $1<p\le q<\infty$ of $q=\infty$ and $p\in(1,\infty)$, then for any $f\in L^p(\Omega)$ it holds
    $$ \left\|\nabla^k e^{\mathbb P\Delta_Dt}\mathbb Pf\right\|_{L^q(\Omega)}\le C(N,p,q,k)t^{-\frac{k}{2}-\frac{N}{2}\left(\frac{1}{p}-\frac{1}{q}\right)}\|f\|_{L^p(\Omega)}\quad t>0. $$
\end{thm}
The proof again follows by the explicit formula of the Green functions for the heat case and from the estimates of \cite{KLT23} for the Stokes semigroup. For the exterior case, we consider only $e^{\Delta_Nt}$:
\begin{thm}\label{t.der.sem.}
Let $N\ge 3$, $p\in[1,\infty]$, let $\Omega\subseteq \mathbb R^N$ be an exterior domain with sufficiently smooth boundary, then 
$$ \left\|\nabla e^{\Delta_Nt}f\right\|_{L^\infty(\Omega)}\le C(\Omega,p) t^{-\frac{1}{2}-\frac{N}{2p}}\|f\|_{L^p(\Omega)}\quad t>0. $$
\end{thm}
The proof follows from \cite{I09} and, as a corollary, we get the $L^p-L^q$ estimate for the gradient:
\begin{cor}\label{c.der.sem.gen.}
    Let $N\ge 3$, $p\in[1,\infty]$, $q\in[2,\infty]$ with $p\le q$ and $\Omega\subseteq\mathbb R^N$ be an exterior domain with sufficiently smooth boundary, then 
    $$ \left\|\nabla e^{\Delta_Nt}f\right\|_{L^q(\Omega)} \le C(\Omega,p,q) t^{-\frac{1}{2}-\frac{N}{2}\left(\frac{1}{p}-\frac{1}{q}\right)}\|f\|_{L^p(\Omega)}\quad t>0. $$
\end{cor}
\begin{proof}\hfill\\
It is easy to see that
$$ \|\nabla e^{\Delta_Nt}f\|_{L^2(\Omega)}=\|(-\Delta_N)^{1/2}e^{\Delta_Nt}f\|_{L^2(\Omega)}. $$
On the other hand 
$$ \|(-\Delta_N)^{1/2}e^{\Delta_Nt}f\|_{L^2(\Omega)}\lesssim t^{-1/2}\|f\|_{L^2(\Omega)}, $$
which comes from Theorem 4.6 of page 101 in \cite{EN06} and Spectral Theorem. Interpolating this estimate with the one of Theorem \ref{t.der.sem.}, we get
$$ \|\nabla e^{\Delta_Nt}f\|_{L^p(\Omega)}\lesssim t^{-\frac{1}{2}}\|f\|_{L^p(\Omega)}\quad p\in[2,\infty]. $$
Finally
$$ \|\nabla e^{\Delta_Nt}f\|_{L^p(\Omega)}\lesssim t^{-\frac{1}{2}}\|e^{\Delta_Nt/2}f\|_{L^p(\Omega)}\lesssim t^{-\frac{1}{2}-\frac{N}{2}\left(\frac{1}{p}-\frac{1}{q}\right)}. $$
\end{proof}

We will not use them, but in the exterior case we can find similar estimates also for the Dirichlet case, but they are weaker:
$$ \|\nabla e^{\Delta_Dt} f\|_{L^q(\Omega)}\lesssim \left\{\begin{array}{ll}
    t^{-\frac{1}{2}}\|f\|_{L^q(\Omega)} & 0<t\le 1 \\
    t^{-\mu}\|f\|_{L^q(\Omega)} & t\ge 1,
\end{array}\right. $$
with
$$ \mu =\left\{\begin{array}{ll}
    \frac{1}{2} & q\in[1,N] \\
    \frac{N}{2q} & q\in(N,\infty].
\end{array}\right. $$
Similar are the ones for the Stokes equation.

\section{Linear Estimates}\label{sec.lin.es.}

In the following, when it is not specified, $\Omega=\mathbb R^N, \mathbb R^N_+$ or $\Omega\subseteq \mathbb R^N$ is an exterior domain with sufficiently smooth boundary where $N\ge 3$.

In system \eqref{EL.sys.eta-ass.} the condition $|\eta +d|=1$ is difficult to be treated. For this reason, we consider the system without the constrain:
\begin{equation}\label{EL.sys.red.}
    \left\{\begin{array}{ll}
       (\partial_t-\Delta)u+\nabla p+u\cdot \nabla u=-{\rm Div}\left(\nabla d\odot\nabla d\right)  & (0,T)\times \Omega \\
       {\rm div}u=0 & (0,T)\times\Omega \\
       (\partial_t-\Delta)d +u\cdot \nabla d = |\nabla d|^2(\eta+d) & (0,T)\times \Omega \\
       u=0, \quad \partial_\nu d=0 & (0,T)\times \partial\Omega \\
       u(0)=u_0, \quad d(0)=d_0 & \Omega. 
    \end{array}\right.     
\end{equation}
In fact, in \cite{HNP16} it is proved the following fact: if $(u,v)$ solves \eqref{EL.sys.} without the condition $|v|=1$ and if $(u_0,v_0)$ are sufficiently regular and $|v_0|=1$, then $|v(t)|=1$ for a.e. $t\in(0,T)$. We will prove with the same argument that also in our case such a property holds. For this reason the study of the system \eqref{EL.sys.red.} is justified. 

Let us consider them the linear system
\begin{equation}\label{lin.sys.}
    \left\{\begin{array}{ll}
        (\partial_t-\Delta)u + \nabla p=f & (0,T)\times \Omega \\
        {\rm div}u=0 & (0,T)\times\Omega \\
        (\partial_t-\Delta)d=g & (0,T)\times \Omega \\
        u=0, \quad \partial_\nu d=0 & (0,T)\times \partial\Omega \\
        u(0)=u_0,\quad d(0)=d_0 & \Omega,
    \end{array}\right.
\end{equation}
and the correspondent one after the projection 
\begin{equation}\label{lin.sys-proj.}
    \left\{\begin{array}{ll}
        (\partial_t-\mathbb P\Delta)u=\mathbb Pf & (0,T)\times \Omega \\
        {\rm div}u=0 & (0,T)\times\Omega \\
        (\partial_t-\Delta)d=g & (0,T)\times \Omega \\
        u=0, \quad \partial_\nu d=0 & (0,T)\times \partial\Omega \\
        u(0)=u_0,\quad d(0)=d_0 & \Omega.
    \end{array}\right.
\end{equation}

Let us consider the problem
$$ \left\{\begin{array}{ll}
    (\partial_t-A)w=f & (0,T)\times\Omega \\
    w(0)=w_0 & \Omega.
\end{array}\right. $$
If $A$ generates a semigroup $e^{At}$, we know that the weak solution can be represented by the Duhamel Formula:
$$ w(t)=e^{At}w_0 + \int_0^t e^{A(t-\tau)}f(\tau)d\tau. $$
We want to prove some smoothing estimates on $w$. More precisely, we want to prove that $w\in X^s_T(\Omega)$ for some $s,T\ge 0$ if $w_0$ and $f$ satisfy some suitable conditions (we recall the definition of $X^s_T(\Omega)$ from \eqref{def.X}). We divided the analysis in the homogeneous term and the inhomogeneous one. The first one is easier:
\begin{lem}\label{l.sm.es.hom.1}
    Let $s\ge 0$, $T\in(0,\infty]$, let $A=\mathbb P\Delta _D,\Delta_N$ and $w_0\in H^s_A(\Omega)$, then $w(t)\in H^{s+1}_A(\Omega)$ for a.e. $t\in (0,T)$ and 
    \begin{equation}\label{sm.es.hom.1}
     \left\|(-A)^{s/2}e^{At}w_0\right\|_{L^\infty((0,T);L^2(\Omega))} + \left\|(-A)^{s/2+1/2}e^{At}w_0\right\|_{L^2((0,T);L^2(\Omega))} \le C(\Omega,s) \|w_0\|_{H^s_A(\Omega)}.
\end{equation}
\end{lem}
\begin{proof}\hfill\\
Let $\varepsilon>0$ and $m\in\mathbb N$ with $m\ge \max\{s,2\}$, then we define
$$ w_0^\varepsilon\coloneqq (1-\varepsilon A)^{-m/2}w_0,\quad  w^\varepsilon(t)= e^{At}w_0^\varepsilon. $$
It is easy to see that 
$$ w^\varepsilon\in C((0,T);H^{s+m}_A(\Omega))\cap C^1((0,T);H^{s+m-2}_A(\Omega)). $$
By definition, $w^\varepsilon$ solves the equation
\begin{equation}\label{approx.lin.sys.}
    \left\{\begin{array}{ll}
    (\partial_t-A)w^\varepsilon=0 & (0,T)\times\Omega \\
    w^\varepsilon(0)=w_0^\varepsilon & \Omega. 
\end{array}\right.
\end{equation}
We notice that
$$ (1-A)^{s/2}w_0^\varepsilon=(1-\varepsilon A)^{-m/2}(1-A)^{s/2}w_0, $$
$$ (1-A)^{s/2}(\partial_t-A)w^\varepsilon= (\partial_t-A)(1-A)^{s/2}w^\varepsilon=(\partial_t-A)(1-\varepsilon A)^{-m/2}e^{At}(1-A)^{s/2}w_0. $$
Therefore, up to take $(1-A)^{s/2}w_0$ in place of $w_0$, we can suppose $s=0$. Then, if we multiply \eqref{approx.lin.sys.} by $w^\varepsilon$, we get that 
$$ \frac{1}{2}\frac{d}{dt}\|w^\varepsilon(\tau)\|_{L^2(\Omega)}^2 + \|(-A)^{1/2}w^\varepsilon(\tau)\|_{L^2(\Omega)}^2= 0\quad \tau\in(0,t). $$
So if we integrate in $(0,t)$ and then we take the $\sup$ over $t\in(0,T)$, we get 
$$ \frac{1}{2}\|w^\varepsilon\|_{L^\infty((0,T);L^2(\Omega))} + \|(-A)^{1/2}w^\varepsilon\|_{L^2((0,T);L^2(\Omega))} =\frac{1}{2}\|w_0^\varepsilon\|_{L^2(\Omega)}. $$
By resolvent estimate
$$ \|w_0^\varepsilon\|_{L^2(\Omega)}\le \|w_0\|_{L^2(\Omega)}. $$
Moreover, $w^\varepsilon(t)\to w(t)$ in $L^2(\Omega)$ as $\varepsilon\to0^+$ for a.e. $t\in(0,T)$. This means that $w(t)\in H^1_A(\Omega)$ for a.e. $t\in (0,T)$ and passing through the limit as $\varepsilon\to0^+$ we conclude.
\end{proof}
The previous estimate gives us many information about the term $e^{At}w_0$. Anyway, as it can be seen from definition \eqref{def.X}, in order to estimate $e^{At}w_0$ in $X^s_T(\Omega)$, we also need to bound the $L^2((0,T);L^2(\Omega))$. For this reason, we need to use the dispersive estimates we introduced in Section \ref{subsec.sem.est.}. In particular, the global case forces us to ask $w_0\in L^1(\Omega)$:
\begin{lem}\label{l.sm.es.hom.2}
    Let $s\in[0,2]$, $T>0$, let $A=\mathbb P\Delta_D,\Delta_N$ and $w_0\in H^s_A(\Omega)$, then $e^{At}w_0\in C((0,T);H^s_A(\Omega))$ and
    \begin{equation}\label{sm.es.hom.2}
     \left\|e^{At}w_0\right\|_{X^s_T(\Omega)} \le C(\Omega,s,T)\|w_0\|_{H^s_A(\Omega)}.
\end{equation}
Moreover, if $T=+\infty$ and in addiction $w_0\in L^1(\Omega)$, then $e^{At}w_0\in C(\mathbb R_+;H^s_A(\Omega))$ and
    \begin{equation}\label{sm.es.hom.3}
     \left\|e^{At}w_0\right\|_{X^s(\Omega)} \le C(\Omega,s) \|w_0\|_{H^s_A\cap L^1(\Omega)}.
\end{equation}
\end{lem}
\begin{proof}\hfill\\
Let us turn back to the estimate 
$$ \|w^\varepsilon\|_{L^\infty((0,T);H^s_A(\Omega))} + \|(-A)^{1/2}w^\varepsilon\|_{L^2((0,T);H^s_A(\Omega))} \lesssim \|w_0^\varepsilon\|_{H^s_A(\Omega)}, $$
where
$$ w_0^\varepsilon\coloneqq (1-\varepsilon A)^{-m/2}w_0,\quad  w^\varepsilon(t)= e^{At}w_0^\varepsilon. $$
By Spectral Theory, it can be seen that 
$$ \|(1-A)^{s/2}g\|_{L^2(\Omega)}\lesssim \|g\|_{L^2(\Omega)} + \|(-A)^{s/2}g\|_{L^2(\Omega)}\quad \forall g\in H^s_A(\Omega). $$
We need to estimate the $L^2((0,T);L^2(\Omega))$-norm of $w^\varepsilon$. On the other hand, by the resolvent estimate
$$ \|w^\varepsilon\|_{L^2((0,T);L^2(\Omega))}\lesssim \|w\|_{L^2((0,T);L^2(\Omega))}. $$
So by Theorem \ref{t.sem.es.0}
$$ \|e^{A t}w_0\|_{L^2(\Omega)}\lesssim \left\{\begin{array}{l}
    \|w_0\|_{L^2(\Omega)}  \\
    t^{-N/4}\|w_0\|_{L^1(\Omega)}.
\end{array}\right. $$
If $T<+\infty$
$$ \|e^{A t}w_0\|_{L^2((0,T);L^2(\Omega))}\le C(T)\|e^{A t}w_0\|_{L^\infty((0,T);L^2(\Omega))}\le C(T)\|w_0\|_{L^2(\Omega)}. $$
On the other hand, when $T=+\infty$ 
$$ \|e^{A t}w_0\|_{L^2((1,\infty);L^2(\Omega))}\lesssim \|w_0\|_{L^1(\Omega)}\|t^{-N/4}\|_{L^2((1,\infty))}\lesssim \|w_0\|_{L^1(\Omega)}. $$
Therefore we conclude thanks to Lemma \ref{l.res.eq.}.
\end{proof}

Let us pass to the inhomogeneus term:
$$ \int_0^t e^{A(t-\tau)}f(\tau)d\tau. $$
In this case, we cannot ask $f(t)\in H^s_A(\Omega)$ for some $s>0$, therefore we need to use a different approach:
\begin{lem}\label{l.sm.es.inh.1}
Let $s\in[0,1]$, $T \in (0,\infty]$, let $A=\Delta_D,\Delta_N,\mathbb P\Delta_D$, let 
$$ f\in L^\frac{2}{2-s}((0,T);L^2(\Omega)), $$
let
$$ w(t)=\int_0^t e^{A(t-\tau)}f(\tau)d\tau, $$
then $w(t)\in H^{s+1}_A(\Omega)$ for a.e. $t\in(0,T)$ and
\begin{equation}\label{sm.es.inh.1}
 \left\|(-A)^{s/2}w\right\|_{L^\infty((0,T);L^2(\Omega))} + \left\|(-A)^{s/2+1/2}w\right\|_{L^2((0,T);L^2(\Omega))} \le C(\Omega,s) \|f\|_{L^\frac{2}{2-s}((0,T);L^2(\Omega))}.
\end{equation}
\end{lem}
In this lemma we considered also the case $A=\Delta_D$ because it will be useful later, when we will try to prove a similar estimate with an higher regularity. 
\begin{proof}\hfill\\
We can suppose $f\in C((0,T);L^2(\Omega))$. Let $\varepsilon>0$ and $m\in\mathbb N$ with $m\ge \max\{s,2\}$, then we define
$$ f^\varepsilon\coloneqq (1-\varepsilon A)^{-m/2}f, \quad w^\varepsilon(t)= \int_0^t e^{A(t-\tau)}f^\varepsilon(\tau)d\tau=(1-\varepsilon A)^{-m/2}w(t). $$
It is easy to see that 
$$ w^\varepsilon\in C((0,T);H^{s+m}_A(\Omega))\cap C^1((0,T);H^{s+m-2}_A(\Omega)). $$
Moreover, by definition, $w^\varepsilon$ solves the equation
\begin{equation}
    \left\{\begin{array}{ll}
    (\partial_t-A)w^\varepsilon=f^\varepsilon & (0,T)\times\Omega \\
    w^\varepsilon(0)=0 & \Omega. 
\end{array}\right.
\end{equation}
Let us prove the estimate for $s=0$: if we multiply by $w^\varepsilon$ and integrate in $(0,t)$ we get as before
$$ \frac{1}{2}\|w^\varepsilon(t)\|_{L^2(\Omega)}^2 + \|(-A)^{1/2}w^\varepsilon\|_{L^2((0,t);L^2(\Omega))}^2 = \int_0^t \left<f^\varepsilon(\tau),w^\varepsilon(\tau)\right>_{L^2(\Omega)}d\tau. $$
On the other hand
$$ \int_0^t \left<f^\varepsilon(\tau),w^\varepsilon(\tau)\right>_{L^2(\Omega)}d\tau\le \|f^\varepsilon\|_{L^1((0,T);L^2(\Omega))}\|w^\varepsilon\|_{L^\infty((0,T);L^2(\Omega))}. $$
So, if we take the $\sup$ over $t\in(0,T)$ we get 
$$ \|w^\varepsilon\|_{L^\infty((0,T);L^2(\Omega))} + \|(-A)^{1/2}w^\varepsilon\|_{L^2((0,T);L^2(\Omega))} \lesssim \|f^\varepsilon\|_{L^1((0,T);L^2(\Omega))}. $$
It holds
$$ w^\varepsilon(t)\xrightarrow{L^2(\Omega)} w(t),\quad f^\varepsilon(t)\xrightarrow{L^2(\Omega)} f(t) \quad \text{for a.e.}\:\:t\in (0,T). $$
So, by the estimate we get that $w(t)\in H^1_A(\Omega)$ for a.e. $t\in(0,T)$ and
$$ \|w\|_{L^\infty((0,T);L^2(\Omega))} + \|(-A)^{1/2}w\|_{L^2((0,T);L^2(\Omega))} \lesssim \|f\|_{L^1((0,T);L^2(\Omega))}. $$
For the case $s=1$, we notice that
$$ \left\{\begin{array}{ll}
    (\partial_t-A)(-A)^{1/2}w^\varepsilon=(-A)^{1/2}f^\varepsilon & (0,T)\times\Omega \\
    (-A)^{1/2}w^\varepsilon(0)=0 & \Omega. 
\end{array}\right. $$
Then, multipling by $(-A)^{1/2}w^\varepsilon$, we get 
$$ \frac{1}{2}\|(-A)^{1/2}w^\varepsilon(t)\|_{L^2(\Omega)}^2 + \|(-A)w^\varepsilon\|_{L^2((0,t);L^2(\Omega))}^2 = \int_0^t \left<(-A)^{1/2}f^\varepsilon(\tau),(-A)^{1/2}w^\varepsilon(\tau)\right>_{L^2(\Omega)}d\tau. $$
This time, we notice that
$$ \left<(-A)^{1/2}f^\varepsilon,(-A)^{1/2}w^\varepsilon\right>=\left<f^\varepsilon,(-A)w^\varepsilon\right>\le \|f^\varepsilon\|_{L^2(\Omega)}\|(-A)w^\varepsilon\|_{L^2(\Omega)}, $$
so by Young's inequality
$$ \|(-A)^{1/2}w^\varepsilon(t)\|_{L^\infty((0,T);L^2(\Omega))} + \|(-A)w^\varepsilon\|_{L^2((0,t);L^2(\Omega))}^2 \lesssim \|f^\varepsilon\|_{L^2((0,T);L^2(\Omega))}. $$
Then $w(t)\in H^2_A(\Omega)$ for a.e. $t\in(0,T)$ and
$$ \|(-A)^{1/2}w\|_{L^\infty((0,T);L^2(\Omega))} + \|(-A)w\|_{L^2((0,T);L^2(\Omega))} \lesssim \|f\|_{L^2((0,T);L^2(\Omega))}. $$
By interpolation we conclude the estimate.
\end{proof}
As for the homogeneuous case, we need to bound the $L^2((0,T);L^2(\Omega))$-norm of the solution using dispersive inequalities. For the global existence, it forces us to ask $f\in L^1(\mathbb R_+;L^1(\Omega))$.
\begin{lem}\label{l.sm.es.inh.2}
Let $s\in[0,1]$, $T>0$, let $A=\mathbb P\Delta_D,\Delta_N,\Delta_D$, let $f$ and $w$ as in Lemma \ref{l.sm.es.inh.1}, then $w\in C((0,T);H^s_A(\Omega))$ and 
\begin{equation}\label{sm.es.inh.2}
\|w\|_{X^s_T(\Omega)}  \le C(T,\Omega,s)\|f\|_{L^\frac{2}{2-s}((0,T);L^2(\Omega))}.
    \end{equation}
Finally, if $T=\infty$ and $f\in L^1(\mathbb R_+;L^1(\Omega))$, then $w\in C(\mathbb R_+;H^s_A(\Omega))$ and 
\begin{equation}\label{sm.es.inh.3}
\|w\|_{X^s(\Omega)} \le C(\Omega,s)\left[ \|f\|_{L^1(\mathbb R_+;L^1(\Omega))} + \|f\|_{L^\frac{2}{2-s}(\mathbb R_+;L^2(\Omega))}\right].
\end{equation}
\end{lem}
\begin{proof}\hfill\\
From the proof of Lemma \ref{l.sm.es.inh.1} we have that     
$$ \|(-A)^{s/2}w^\varepsilon\|_{L^\infty((0,T);L^2(\Omega))} + \|(-A)^{s/2+1/2}w^\varepsilon\|_{L^2((0,T);L^2(\Omega))} \lesssim \|f^\varepsilon\|_{L^\frac{2}{2-s}((0,T);L^2(\Omega))}, $$
where $w^\varepsilon$ and $f^\varepsilon$ are the same approximation we used in the previous proof. Now we want to estimate the $L^2(\Omega)$-norm of $w^\varepsilon$: by resolvent estimate
$$ \|w^\varepsilon(t)\|_{L^2(\Omega)}\lesssim \|w(t)\|_{L^2(\Omega)}\quad t\in(0,T). $$
When $T<+\infty$, then by Theorem \ref{t.sem.es.0} we have
    $$ \left\|\int_0^t e^{A(t-\tau)}f(\tau)d\tau\right\|_{L^2(\Omega)}\le \int_0^t \|f(\tau)\|_{L^2(\Omega)}d\tau\le t^\frac{s}{2} \|f\|_{L^\frac{2}{2-s}((0,T);L^2(\Omega))}.$$ 
So, 
$$ \left\|\int_0^t e^{A (t-\tau)}f(\tau)d\tau\right\|_{L^2((0,T);L^2(\Omega))}\le $$
$$ \le C(T)\left\|\int_0^t e^{A(t-\tau)}f(\tau)d\tau\right\|_{L^\infty((0,T);L^2(\Omega))}\le C(T)\|f\|_{L^\frac{2}{2-s}((0,T);L^2(\Omega))}. $$
Let $T=+\infty$ now and let us consider only the case $t\ge 1$ (when $t\le 1$ we can repeat the previous argument). We split the integral in two pieces: 
$$ \int_0^t e^{A(t-\tau)}f(\tau)d\tau = \int_0^{t-1} e^{A(t-\tau)}f(\tau)d\tau + \int_{t-1}^t e^{A(t-\tau)}f(\tau)d\tau. $$
So
$$ \left\|\int_0^{t-1} e^{A (t-\tau)}f(\tau)d\tau\right\|_{L^2(\Omega)}\lesssim \int_0^{t-1}(t-\tau)^{-\frac{N}{4}}\|f(\tau)\|_{L^1(\Omega)}d\tau, $$
and 
$$ \left\|\int_{t-1}^t e^{A(t-\tau)}f(\tau)d\tau\right\|_{L^2(\Omega)}\lesssim \int_{t-1}^t\|f(\tau)\|_{L^2(\Omega)}d\tau. $$
So, by Young's Inequality
$$ \left\|\int_0^{t-1} e^{A(t-\tau)}f(\tau)d\tau\right\|_{L^\infty(\mathbb R_+;L^2(\Omega))}\lesssim $$
$$ \lesssim \left\|t^{-\frac{N}{4}}\mathbbm{1}_{(1,\infty)}(t)\right\|_{L^\infty(\mathbb R_+)}\|f\|_{L^1(\mathbb R_+;L^1(\Omega))}\lesssim \|f\|_{L^1(\mathbb R_+;L^1(\Omega))}, $$
$$ \left\|\int_0^{t-1} e^{A(t-\tau)}f(\tau)d\tau\right\|_{L^2(\mathbb R_+;L^2(\Omega))}\lesssim \left\|t^{-\frac{N}{4}}\mathbbm{1}_{(1,\infty)}(t)\right\|_{L^2(\mathbb R_+)}\|f\|_{L^1(\mathbb R_+;L^1(\Omega))}\lesssim \|f\|_{L^1(\mathbb R_+;L^1(\Omega))}, $$
and
$$ \left\|\int_{t-1}^t e^{A(t-\tau)}f(\tau)d\tau\right\|_{L^\infty(\mathbb R_+;L^2(\Omega))}\lesssim \left\|\mathbbm{1}_{(0,1)}(t)\right\|_{L^\frac{2}{s}(\mathbb R_+)}\|f\|_{L^\frac{2}{2-s}(\mathbb R_+;L^2(\Omega))}\lesssim \|f\|_{L^\frac{2}{2-s}(\mathbb R_+;L^2(\Omega))}, $$
$$ \left\|\int_{t-1}^t e^{A(t-\tau)}f(\tau)d\tau\right\|_{L^2(\mathbb R_+;L^2(\Omega))}\lesssim $$
$$ \lesssim \left\|\mathbbm{1}_{(0,1)}(t)\right\|_{L^\frac{2}{1+s}(\mathbb R_+)}\|f\|_{L^\frac{2}{2-s}(\mathbb R_+;L^2(\Omega))}\lesssim \|f\|_{L^\frac{2}{2-s}(\mathbb R_+;L^2(\Omega))}. $$
Then, thanks to Lemma \ref{l.res.eq.}, we get
$$ \left\|w^\varepsilon\right\|_{X^s_T(\Omega)} \le \left\{ \begin{array}{ll}
    C(T) \|f\|_{L^\frac{2}{2-s}((0,T);L^2(\Omega))} & T<\infty \\
    C\left[\|f\|_{L^1(\mathbb R_+;L^1(\Omega))} + \|f\|_{L^\frac{2}{2-s}(\mathbb R_+;L^2(\Omega))}\right] & T=\infty.
\end{array}\right. $$
Passing by the limit as $\varepsilon\to 0^+$
$$ \left\|w\right\|_{X^s_T(\Omega)} \le \left\{ \begin{array}{ll}
    C(T) \|f\|_{L^\frac{2}{2-s}((0,T);L^2(\Omega))} & T<\infty \\
    C\left[\|f\|_{L^1(\mathbb R_+;L^1(\Omega))} + \|f\|_{L^\frac{2}{2-s}(\mathbb R_+;L^2(\Omega))}\right] & T=\infty.
\end{array}\right. $$
Moreover, $w^\varepsilon\in C((0,T);H^s_A(\Omega))$ so $w\in C((0,T);H^s_A(\Omega))$. 
\end{proof}
In order to get more regularity, we need to distinguish each choice of $\Omega$: in $\mathbb R^N$ we can get easily the smoothing estimate for all $s\ge 0$:
\begin{lem}\label{l.sm.es.RN.}
    Let $N\ge 3$, $s\ge 0$, $T>0$, let $f\in L^2((0,T);H^s(\mathbb R^N))$, let $w\in X^{1}_T(\mathbb R^N)$ be the function from Lemma \ref{l.sm.es.inh.2}, then for $T<\infty$ we have that $w\in X^{s+1}_T(\mathbb R^N)$ with
    \begin{equation}\label{sm.es.RN.1}
    \left\|w\right\|_{X^{s+1}_T(\mathbb R^N)}  \le C(T,N,s)\|f\|_{L^2((0,T);H^s(\mathbb R^N))}.
    \end{equation}
Finally, if $T=\infty$ and $f\in L^1(\mathbb R_+;L^1(\mathbb R^N))$, the same result holds with 
\begin{equation}\label{sm.es.RN.2}
\left\|w\right\|_{X^{s+1}(\mathbb R^N)} \le C(N,s)\left[|f\|_{L^1(\mathbb R_+;L^1(\mathbb R^N))} + \|f\|_{L^2(\mathbb R_+;H^s(\mathbb R^N))}\right] .
\end{equation}
\end{lem}
The proof follows from the commutation between the derivatives and the semi groups. In the half-space we are able to apply the same argument only for the heat equation with Neumann boundary conditions:
\begin{lem}\label{l.sm.es.hs.}
    Let $N\ge 3$, $s\in[0,1]$, $T>0$, let $f\in L^2((0,T);H^s(\mathbb R^N_+))$, let $w\in X^1_T(\mathbb R^N_+)$ be the function from Lemma \ref{l.sm.es.inh.2} for $A=\Delta_N$, then for $T<\infty$ we have that $w\in X^{s+1}_T(\mathbb R^N_+)$ with
    \begin{equation}\label{sm.es.hs.1}
    \left\|w\right\|_{X^{s+1}_T(\mathbb R^N_+)}  \le C(T,N,s)\|f\|_{L^2((0,T);H^s(\mathbb R^N_+))}.
    \end{equation}
Finally, if $T=\infty$ and $f\in L^1(\mathbb R_+;L^1(\mathbb R^N_+))$, the same result holds with 
\begin{equation}\label{sm.es.hs.2}
\left\|w\right\|_{X^{s+1}(\mathbb R^N_+)} \le C(N,s)\left[  \|f\|_{L^1(\mathbb R_+;L^1(\mathbb R^N_+))} + \|f\|_{L^2(\mathbb R_+;H^s(\mathbb R^N_+))}\right] .
\end{equation}
\end{lem}
\begin{proof}\hfill\\
It is sufficient to notice that 
$$ \partial_je^{\Delta_Nt} = e^{\Delta_Nt}\partial_j \quad j=1,\ldots, N-1,\quad \partial_Ne^{\Delta_Nt}=e^{\Delta_Dt}\partial_N. $$
Then we can apply for $\nabla w$ Lemma \ref{l.sm.es.inh.1} in order to conclude.
\end{proof}
Again, in the exterior case, we get the smoothing estimate only for the heat equation with Neumann boundary conditions. The strategy is to apply the Tubular Neighbourhood Theorem as a special parametrization for $\Omega$ near the boundary. In this way, we can trasfer the problem in the half-space, where we can apply an argument similar to the previous one.
\begin{lem}\label{l.sm.es.ex.}
    Let $N\ge 3$, let $\Omega\subseteq \mathbb R^N$ be an exterior domain with sufficiently smooth boundary, let $T>0$, let $f\in L^2((0,T);H^s(\Omega))$, let $w\in X^1_T(\Omega)$ be the function from Lemma \ref{l.sm.es.inh.2} for $A=\Delta_N$, then for $T<\infty$ we have that $w\in X^{s+1}_T(\Omega)$ with
    \begin{equation}\label{sm.es.ex.1}
    \left\|w\right\|_{X^{s+1}_T(\Omega)}  \le C(T,\Omega,s)\|f\|_{L^2((0,T);H^s(\Omega))}.
    \end{equation}
Finally, if $T=\infty$ and $f\in L^1(\mathbb R_+;L^1(\Omega))$, the same result holds with
\begin{equation}\label{sm.es.ex.2}
\left\|w\right\|_{X^{s+1}(\Omega)} \le C(\Omega,s)\left[  \|f\|_{L^1(\mathbb R_+;L^1(\Omega))} + \|f\|_{L^2(\mathbb R_+;H^s(\Omega))} \right].
\end{equation}
\end{lem}
\begin{proof}\hfill\\
Let $\varphi\in C^\infty([0,\infty))$ such that 
$$ \varphi(r)=\left\{\begin{array}{ll}
    1 & r\in [0,r_0]\\
    0 & r\ge 2r_0,
\end{array}\right. $$
where $r_0>0$ will be choose properly later. Let us call $d(x)\coloneqq d(x,\partial\Omega)$. From \cite{PS16}, we know that $d$ is a $C^2$ function. In particular, we consider
$$ w(t,x)= w(t,x)\varphi(d(x)) + w(t,x)(1-\varphi(d(x))). $$
If we denote $E_0[g]$ the 0-extension of a function $g$ in $\mathbb R^N$, then 
$$ \left\{\begin{array}{ll}
    (\partial_t-\Delta)E_0[w(1-\varphi\circ d)]= E_0[f(1-\varphi\circ d) - 2\nabla w\cdot \nabla (\varphi\circ d) - w\Delta(\varphi\circ d)] & (0,T)\times\mathbb R^N \\
    E_0[w(1-\varphi\circ d)](0)= E_0[w_0(1-\varphi\circ d)] & \mathbb R^N.
\end{array}\right. $$
So we get the estimate using the smoothing lemma for $\mathbb R^N$. Let us consider now the function $w\varphi\circ d$. As it can be seen by \cite{PS16}, if we take 
$$ r_0\le \frac{1}{\max\{\kappa_i(y)\mid y\in\partial\Omega, \:\:i=1,\ldots, N-1\}}, $$
where $\kappa_i$ are the principal curvatures of $\partial\Omega$, then each point $x$ in the tubular neighbourhood of $\partial\Omega$ with distance $r_0$ can be written as
\begin{equation} \label{eq.noc90}
x=\phi(y,r)=y + r\nu(y), 
\end{equation}
where $y\in\partial\Omega$ and $r=d(x)$. Up to consider a partition of unity of $\partial\Omega$ and proceed as in the previous case, we can suppose $y=\psi(\theta)$ with $\psi\in C^4(B_\delta;\partial\Omega)$, where $B_\delta$ is a ball of $\mathbb R^{N-1}$ of radius $\delta>0$.  In \cite{PS16} they take $\Omega$ of class $C^2$ and then prove $\phi$ to be a $C^1$-diffeomorphism if we take $r_0$ sufficiently small. They did it by the Inverse Function Theorem so, since in our case $\phi\in C^3$, we get that $\phi$ is a $C^3$-diffeomorphism. It is well-known that
$$ \Delta =  \partial_r^2 + (N-1)H_{\Gamma_r}\partial_r + \Delta_{\Gamma_r}, $$
where $\Gamma_r$ is the surface parallel to $\partial\Omega$ with distance $r$, $\Delta_{\Gamma_r}$ is the Laplace-Beltrami Operator on the surface $\Gamma_r$ and $H_{\Gamma_r}$ is the mean curvature of $\Gamma_r$. Finally, we notice that 
$$ \partial_{\nu}w=\partial_r w\circ \phi. $$
If we call $v(y,r)\coloneqq w(\phi(y,r))$, then $w$ solves
$$ \left\{ \begin{array}{ll}
     (\partial_t-\Delta_{\Gamma_r} - \partial_r^2 - H_{\Gamma_r}\partial_r) v = f\circ \phi & (0,T)\times \partial\Omega \times(0,r_0) \\
     \partial_rv=0 & (0,T)\times \partial\Omega\times\{0\} \\
     v(0)= u\circ \phi & \partial\Omega\times(0,r_0).
\end{array}\right. $$
Now we can take the derivatives in $y$ (by the parametrization of $\partial\Omega)$ and $r$: It is well-known that for any $\Sigma$ Hypersurfaces the Laplace-Beltrami Operator can be written as 
$$ \Delta_{\Sigma}=g^{ij}(\Sigma)\left(\partial_{ij}-\Lambda^k_{ij}(\Sigma)\partial_k\right), $$
where $G=(g_{ij})$ is the first fundamental form of $\Sigma$ and $(g^{ij})=G^{-1}$ and where $\Lambda^k_{ij}$ are the Christoffel Symbols of $\Sigma$. Therefore $\nabla_y v$ solves
$$ \left\{ \begin{array}{ll}
     (\partial_t-\Delta_{\Gamma_r} - \partial_r^2 - H_{\Gamma_r}\partial_r) \nabla_yv = \nabla_y(f\circ \phi) + R_1(y,r,\nabla_{y,r}v,\nabla_{y,r}^2v) & (0,T)\times \partial\Omega \times(0,r_0) \\
     \partial_r\nabla_yv=0 & (0,T)\times \partial\Omega\times\{0\} \\
     \nabla_yv(0)= 0 & \partial\Omega\times(0,r_0),
\end{array}\right. $$
where 
$$ R_1(y,r,\nabla_{y,r}v,\nabla_{y,r}^2v)=\nabla_yg^{ij}(\Gamma_r)\left(\partial_{ij}-\Lambda^k_{ij}(\Gamma_r)\partial_k\right)v - g^{ij}(\Gamma_r)\nabla_y\Lambda^k_{ij}(\Gamma_r)\partial_k v - \nabla_yH(\Gamma_r)\partial_rv.  $$
On the other hand 
$$ \left\{ \begin{array}{ll}
     (\partial_t-\Delta_{\Gamma_r} - \partial_r^2 - H_{\Gamma_r}\partial_r) \partial_rv = \partial_r(f\circ \phi) + R_2(r,y,\nabla_{y,r}v,\nabla_{y,r}^2v) & (0,T)\times \partial\Omega \times(0,r_0) \\
     \partial_rv=0 & (0,T)\times \partial\Omega\times\{0\} \\
     \partial_rv(0)= 0 & \partial\Omega\times(0,r_0),
\end{array}\right. $$
where 
$$ R_2(y,r,\nabla_{y,r}v,\nabla_{y,r}^2v)=\partial_rg^{ij}(\Gamma_r)\left(\partial_{ij}-\Lambda^k_{ij}(\Gamma_r)\partial_k\right)v - g^{ij}(\Gamma_r)\partial_r\Lambda^k_{ij}(\Gamma_r)\partial_k v - \partial_rH(\Gamma_r)\partial_rv.  $$
Finally, if we turn back to the equation in $\Omega$ and we apply Lemma \ref{l.sm.es.inh.1}, we get that $\nabla_{y,r}v\circ\phi^{-1}(\tau)\in D((-\Delta_N)^{s+1/2})$ and
$$ \|(-\Delta_N)^{s/2}\nabla_{y,r}v\circ\phi^{-1}\|_{L^\infty((0,T);L^2(\Omega))} + \|(-\Delta_N)^{s+1/2}\nabla_{y,r}v\circ\phi^{-1}\|_{L^2((0,T);L^2(\Omega))} \lesssim $$
$$ \lesssim \|\nabla_{y,r}(f\circ\phi)\circ\phi^{-1}\|_{L^2((0,T);L^2(\Omega))} + \|R_1 + R_2 \|_{L^2((0,T);L^2(\Omega))} + \|\nabla_{y,r}v\circ\phi^{-1}\|_{L^2((0,T);L^2(\Omega))}. $$
By the regularity of $\phi$, $R_1$ and $R_2$ and from Lemma \ref{l.sm.es.inh.2} we get
$$ \|\nabla w\|_{X^s_T}\lesssim \|\nabla_{y,r}v\circ\phi^{-1}\|_{X^s_T}\lesssim $$
$$ \lesssim \|(1-\Delta_N)^{s/2}\nabla_{y,r}v\circ\phi^{-1}\|_{L^\infty((0,T);L^2(\Omega))} + \|(1-\Delta_N)^{s+1/2}\nabla_{y,r}v\circ\phi^{-1}\|_{L^2((0,T);L^2(\Omega))} \lesssim $$
$$ \lesssim \|f\|_{L^1(\mathbb R_+;L^1(\Omega))} + \|f\|_{L^2(\mathbb R_+;H^1(\Omega))}, $$
where in the second estimate we used Lemma \ref{l.res.eq.}. The proof is then concluded.
\end{proof}

Finally we can prove the linear estimates for the systems \eqref{lin.sys.} and \eqref{lin.sys-proj.}. In $\mathbb R^N$ we can use a particular structure of the system: in the Stokes equation of \eqref{EL.sys.red.}, the nonlinearity can be rewritten as 
$$ -u\cdot \nabla u - {\rm Div}(\nabla d\odot \nabla d)={\rm Div}\left(-u\otimes u - \nabla d\odot \nabla d\right), $$
where 
$$ (w\otimes z)_{i,j}=w_iz_j\quad i,j=1,\ldots, N,\quad z,w\colon\mathbb R^N\to\mathbb R^N. $$
This because ${\rm div}u=0$. So we can rewrite the function $f$ in the linear systems \eqref{lin.sys.} and \eqref{lin.sys-proj.} as $f={\rm Div}F$ for some $F\colon(0,T)\times\mathbb R^N\to\mathbb R^{N^2}$.
\begin{thm}\label{t.lin.ex.RN}
Let $N\ge 3$, $s\ge 0$, let 
$$ u_0\in H^s\left(\mathbb{R}^N;\mathbb{R}^N\right), \quad d_0\in H^{s+1}\left(\mathbb{R}^N;\mathbb R^N \right), $$
let $T>0$ and
$$ F\in L^2\left((0,T);H^s\left(\mathbb{R}^N;\mathbb{R}^{N^2}\right)\right), \quad g\in L^2\left((0,T);H^{s}\left(\mathbb{R}^N\right)\right), $$
then there is a solution
$$ (u,p,d)\in X^s_T\left(\mathbb R^N\right)\times L^2\left((0,T);L^2_{loc}\left(\mathbb R^N\right)\right) \times X^{s+1}_T\left(\mathbb R^N\right), $$
unique up to additive functions $\rho(t)$ in the pressure terms,
for \eqref{lin.sys.} in $(0,T)$ with $f={\rm Div}F$, 
where $X^s_T(\mathbb R^N)$ is defined in \eqref{def.X} and 
$$  \|u\|_{X^s_T(\mathbb R^N)} + \|d\|_{X^{s+1}_T(\mathbb R^N)} \le C(N,s,T)\left[\|u_0\|_{H^s(\mathbb{R}^N)}+\|d_0\|_{H^{s+1}(\mathbb{R}^N)} + \right.$$
$$ \left. + \|F\|_{L^2((0,T);H^s(\mathbb{R}^N))} + \|g\|_{L^2((0,T);H^s(\mathbb{R}^N))}\right]. $$
Moreover, if $T=\infty$ and in addiction
$$ u_0\in L^1\left(\mathbb R^N;\mathbb R^N\right),\quad d_0\in L^1\left(\mathbb R^N;\mathbb R^N\right), $$
$$ g\in L^1\left(\mathbb R_+;L^1\left(\mathbb R^N;\mathbb R^N\right)\right), \quad {\rm Div}F\in L^1\left((0,T);L^1\left(\mathbb R^N;\mathbb R^N\right)\right), $$
then the result is still true with 
$$  \|u\|_{X^s(\mathbb R^N)} + \|d\|_{X^{s+1}(\mathbb R^N)}\le C(N,s)\left[\|u_0\|_{H^s\cap L^1(\mathbb{R}^N)}+\|d_0\|_{H^{s+1}\cap L^1(\mathbb{R}^N)} + \right.$$
$$ \left. + \|({\rm Div}F,g)\|_{L^1(\mathbb R_+;L^1(\mathbb R^N))} + \|F\|_{L^2(\mathbb R_+;H^s(\mathbb{R}^N))}  + \|g\|_{L^2(\mathbb R_+;H^s(\mathbb{R}^N))}\right]. $$
\end{thm}
\begin{proof}\hfill\\
The estimate for $d$ and for $u$ when $s\ge 1$ follows by Lemma \ref{l.sm.es.RN.}. On the other hand, if we turn back to the proof of Lemma \ref{l.sm.es.inh.1}, in the case $s=0$ we take 
$$ (u^\varepsilon,u_0^\varepsilon,f^\varepsilon)\coloneqq (1-\varepsilon\mathbb P\Delta)^{-m/2} (u,u_0,{\rm Div}F), $$
for $m\in\mathbb N$ sufficiently large. In this case, we multiply the equation for $u^\varepsilon$ by $u^\varepsilon$ and we get
$$ \frac{1}{2}\frac{d}{dt}\|u^\varepsilon(t)\|_{L^2(\mathbb R^N)}^2 + \|(-\mathbb P\Delta)^{1/2} u^\varepsilon(t)\|_{L^2(\mathbb R^N)}^2 = \frac{1}{2}\|u_0^\varepsilon\|_{L^2(\mathbb R^N)}^2 + \left<\mathbb Pf^\varepsilon(t),u^\varepsilon(t)\right>_{L^2(\mathbb R^N)}. $$
Firstly we notice that $\mathbb Pf^\varepsilon=f^\varepsilon$. Moreover
$$ \left<f^\varepsilon(t),u^\varepsilon(t)\right>_{L^2(\mathbb R^N)}=\left<(1-\varepsilon\mathbb P\Delta)^{-m/2}{\rm Div}F(t),(1-\varepsilon\mathbb P\Delta)^{-m/2}u(t)\right>_{L^2(\mathbb R^N)} = $$
$$ = \left<F(t),\nabla (1-\varepsilon \mathbb P\Delta )^{-m}u(t)\right>_{L^2(\mathbb R^N)}\le \|F(t)\|_{L^2(\mathbb R^N)}\|\nabla (1-\varepsilon \mathbb P\Delta )^{-m/2}u^\varepsilon(t)\|_{L^2(\mathbb R^N)}. $$
If we call $w^\varepsilon=(1-\varepsilon\mathbb P\Delta)^{-m/2}u^\varepsilon$, then
$$ \|(-\mathbb P\Delta)^{1/2}w^\varepsilon(t)\|_{L^2(\mathbb R^N)}^2=\left<(-\mathbb P\Delta)^{1/2}w^\varepsilon(t),(-\mathbb P\Delta)^{1/2}w^\varepsilon(t)\right>_{L^2(\mathbb R^N)} = $$
$$ \left<w^\varepsilon(t),(-\mathbb P\Delta)w^\varepsilon(t)\right>_{L^2(\mathbb R^N)}= - \left<w^\varepsilon(t),\Delta w^\varepsilon(t)\right>_{L^2(\mathbb R^N)}=\|\nabla w^\varepsilon(t)\|_{L^2(\mathbb R^N)}^2. $$
So
$$ \|\nabla (1-\varepsilon \mathbb P\Delta )^{-m/2}u^\varepsilon(t)\|_{L^2(\mathbb R^N)}= $$
$$ = \|(-\mathbb P\Delta)^{1/2}(1-\varepsilon \mathbb P\Delta )^{-m/2}u^\varepsilon(t)\|_{L^2(\mathbb R^N)}\lesssim \|(-\mathbb P\Delta)^{1/2}u^\varepsilon(t)\|_{L^2(\mathbb R^N)}, $$
and therefore
$$ \left<f^\varepsilon(t),u^\varepsilon(t)\right>_{L^2(\mathbb R^N)}\le \|F(t)\|_{L^2(\mathbb R^N)}\|(-\mathbb P\Delta)^{1/2}u^\varepsilon(t)\|_{L^2(\mathbb R^N)}. $$
Proceeding as before we get then 
$$ \|u\|_{X^0_T(\mathbb R^N)}\lesssim \|u_0\|_{L^2(\mathbb R^N)} + \|F\|_{L^2((0,T);L^2(\mathbb R^N))}. $$
By interpolation with the result for $s\ge 1$ we conclude the existence of a solution $(u,d)\in X^s_T(\mathbb R^N)\times X^{s+1}_T(\mathbb R^N)$ for the system \eqref{EL.sys-proj.}. For the pressure part, it is sufficient to apply Lemma \ref{l.ex.press.} in order to conclude. 
\end{proof}
Finally, we can also prove the linear estimate for the linear system in the half-plane and in the exterior case:
\begin{thm}\label{t.lin.ex.dom.}
Let $s\in[0,1]$, let $\Omega=\mathbb R^N_+$ or $\Omega\subseteq \mathbb R^N$ be an exterior domain with sufficiently smooth boundary, let 
$$ u_0\in H^s_{\mathbb P\Delta_D}\left(\Omega;\mathbb R^N\right), \quad d_0 \in H^{s+1}_{\Delta_N}\left(\Omega;\mathbb R^N\right),  $$
let $T>0$ and let 
$$ F\in L^2\left((0,T);L^2\left(\Omega;\mathbb R^{N^2}\right)\right),\quad {\rm Div}F\in L^\frac{2}{2-s}\left((0,T);L^2\left(\Omega;\mathbb R^N\right)\right), $$
$$ g\in L^2\left((0,T);H^s\left(\Omega;\mathbb R^N\right)\right), $$
then there is a solution
$$ (u,p,d)\in X^s_T(\Omega)\times L^2_{loc}\left((0,T);L^2_{loc}(\Omega)\right)\times X^{s+1}_T(\Omega), $$ 
unique up to additive functions $\rho(t)$ in the pressure term, for the linear system \eqref{lin.sys.}, where $X^s_T(\Omega)$ is defined in \eqref{def.X}, and 
$$  \|u\|_{X^s_T(\Omega)}+ \|d\|_{X^{s+1}_T(\Omega)} \le  C(T,\Omega,s)\left[\|u_0\|_{H^s(\Omega)} + \|d_0\|_{H^{s+1}\Omega)}\right] +  $$
$$ +C(\Omega,s)\left[ \|{\rm Div}F\|_{L^\frac{2}{2-s}((0,T);L^2(\Omega))} + \|g\|_{L^2((0,T);H^s(\Omega))}\right]. $$
Moreover, if $T=\infty$ and, in addiction
$$ u_0,d_0\in  L^1\left(\Omega;\mathbb R^N\right),\quad {\rm Div}F,g\in L^1\left((0,T);L^1\left(\Omega;\mathbb R^N\right)\right), $$
then it holds the same result with 
$$  \|u\|_{X^s(\Omega)}+ \|d\|_{X^{s+1}(\Omega)} \le C(\Omega,s)\left[ \|u_0\|_{L^1\cap H^s(\Omega)} + \|d_0\|_{L^1\cap H^{1}(\Omega)} \right]+   $$
$$  + C(\Omega,s)\left[] \|({\rm Div}F,g)\|_{L^1(\mathbb R_+;L^1(\Omega))} + \|{\rm Div}F\|_{L^\frac{2}{2-s}(\mathbb R_+;L^2(\Omega))} + \|g\|_{L^2(\mathbb R_+;H^s(\Omega))}\right]. $$
\end{thm}
\begin{proof}\hfill\\
The estimate for the homogeneous part follows from Lemma \ref{l.sm.es.hom.2}, while the estimate for the inhomogeneous part follows from Lemma \ref{l.sm.es.inh.2} for $u$ and from Lemmas \ref{l.sm.es.hs.} and \ref{l.sm.es.ex.} for $d$. 
\end{proof}

\section{Local Existence}\label{sec.loc.ex.}

We focus on the local reduced system
\begin{equation}\label{sys.loc-red.}
    \left\{\begin{array}{ll}
       (\partial_t-\mathbb P\Delta)u= - \mathbb P\left(u\cdot \nabla u+{\rm Div}\left(\nabla d\odot\nabla d\right)\right)  & (0,T)\times \Omega \\
       {\rm div}u=0 & (0,T)\times\Omega \\
       (\partial_t-\Delta)d +u\cdot \nabla d = |\nabla d|^2(d+\eta) & (0,T)\times\Omega \\
       \partial_\nu d=0,\quad u=0 & (0,T)\times \partial\Omega \\
       u(0)=u_0,\quad d(0)=d_0 & \Omega,
    \end{array}\right.     
\end{equation}
for $d_0=v_0-\eta$. Again, if it is not specified, $\Omega=\mathbb R^N,\mathbb R^N_+$ or an exterior sufficiently smooth with $N\ge 3$. We are going to use a classical contraction argument to prove the local existence in $X^s_T(\Omega)$. 

\subsection{Leibniz Rule in $H^s(\Omega)$ and Polylinear Local Estimates}

The first task is to estimate the nonlinearities which arise from \eqref{sys.loc-red.} using Theorems \ref{t.lin.ex.RN} and \ref{t.lin.ex.dom.}. This part is in common for every choice of $\Omega$. In order to get it, we need a Leibniz estimate in $H^s(\Omega)$. Leibniz estimates can be found in \cite{I23}, \cite{IMT18}, \cite{GT19}, \cite{FGO18}. This approach is based on functional calculus of self-adjoint operators associated with $\Omega$, therefore they need specific boundary conditions. Anyway, we need the estimate in Sobolev spaces without boundary conditions, a variant of the previous results for the local existence:
\begin{lem}\label{l.Leib.1}
    Let $s>\frac{N}{2}-1$, let  $f,g\in H^{s+1}(\Omega)$, then we can find $\theta=\theta(s)\in(0,1]$ such that
    $$ \|fg\|_{H^s(\Omega)}\le C(\Omega,s)\left[\|f\|_{H^s(\Omega)}^\theta\|f\|_{H^{s+1}(\Omega)}^{1-\theta}\|g\|_{H^s(\Omega)} + \|g\|_{H^s(\Omega)}^\theta\|g\|_{H^{s+1}(\Omega)}^{1-\theta}\|f\|_{H^s(\Omega)}\right]. $$
\end{lem}
\begin{proof}\hfill\\
It is well-known that $H^s(\Omega)$ is an algebra for $s>\frac{N}{2}$ (see Theorem 4.39 of \cite{AF03}), so we consider only $s\le \frac{N}{2}$. Let us suppose at the beginning $s\in\mathbb N$ and, consequently, $s\ge 1$. Firstly we notice that 
$$ \|fg\|_{L^2(\Omega)}\le \|f\|_{L^\frac{2N}{N-2}(\Omega)}\|g\|_{L^N(\Omega)}\lesssim \|f\|_{H^1(\Omega)}\|g\|_{H^s(\Omega)}\le \|f\|_{H^s(\Omega)}\|g\|_{H^s(\Omega)}, $$
where we used Sobolev embeddings with the condition $s>\frac{N}{2}-1$. Let now $\beta\in\mathbb N^N$ with $|\beta|=s$, then by the standard Leibniz formula
\begin{equation}\label{proof.Leib.}
    \|D^\beta(fg)\|_{L^2(\Omega)}\le \|D^\beta f g\|_{L^2(\Omega)} + \|fD^\beta g\|_{L^2(\Omega)} + \sum_{\gamma_1+\gamma_2=\beta,\:|\gamma_j|\ge 1} \|D^{\gamma_1}fD^{\gamma_2}g\|_{L^2(\Omega)}.
\end{equation} 
Let us see the first term: if $s<\frac{N}{2}$ then
$$ \|D^\beta f g\|_{L^2(\Omega)}\le \|D^\beta f\|_{L^\frac{N}{s}(\Omega)}\|g\|_{L^\frac{2N}{N-2s}(\Omega)}.  $$
On the other hand, if $s=\frac{N}{2}$, we can take $\varepsilon\in\left(0,\frac{2}{N}\right)$ and 
$$ \|D^\beta f g\|_{L^2(\Omega)}\le \|D^\beta f\|_{L^\frac{2}{1-\varepsilon}(\Omega)}\|g\|_{L^\frac{2}{\varepsilon}(\Omega)}. $$
In both case, by Sobolev embedding, we have $p\in\left(2,\frac{2N}{N-2}\right)$ and $q\in(1,\infty)$ such that $L^q(\Omega)\hookrightarrow H^s(\Omega)$ such that 
    $$ \|D^\beta f g\|_{L^2(\Omega)}\le \|D^\beta f\|_{L^p(\Omega)}\|g\|_{L^q(\Omega)}\lesssim $$
    $$ \lesssim \|D^\beta f\|_{L^2(\Omega)}^\theta \|D^\beta f \|_{L^\frac{2N}{N-2}(\Omega)}^{1-\theta}\|g\|_{H^s(\Omega)}\lesssim  \|f\|_{H^s(\Omega)}^\theta \|f\|_{H^{s+1}(\Omega)}^{1-\theta}\|g\|_{H^s(\Omega)}, $$
    for some $\theta\in(0,1)$.

The estimate for second term is the same with $g$ in place of $f$. For the last term of \eqref{proof.Leib.}, we notice that 
$$ \frac{N-2(s+1-|\gamma_1|)}{2N} + \frac{N-2(s-|\gamma_2|)}{2N}=1-\frac{s+1}{N}<\frac{1}{2}, $$
so we can find $p,q\in(2,\infty)$ and $r<s+1$ such that 
$$ \frac{1}{p} + \frac{1}{q}=\frac{1}{2}, \quad L^p(\Omega)\hookrightarrow H^{r-|\gamma_1|}(\Omega),\quad L^q(\Omega)\hookrightarrow H^{s-|\gamma_2|}(\Omega). $$
So
$$ \|D^{\gamma_1}fD^{\gamma_2}g\|_{L^2(\Omega)}\le \|D^{\gamma_1}f\|_{L^p(\Omega)}\|D^{\gamma_2}g\|_{L^q(\Omega)}\lesssim $$
$$ \lesssim \|f\|_{H^r(\Omega)}\|g\|_{H^s(\Omega)}\lesssim \|f\|^\theta_{H^s(\Omega)}\|f\|_{H^{s+1}(\Omega)}^{1-\theta}\|g\|_{H^s(\Omega)}, $$
for some $\theta\in (0,1)$. Let us consider now the case $s\in(0,1)$, then from \eqref{Hs-norm.}
$$ \|fg\|_{H^s(\Omega)}^2 = \|fg\|_{L^2(\Omega)}^2 + \int_\Omega\int_\Omega \frac{|(fg)(x) - (fg)(y)|^2}{|x-y|^{2s+N}}dxdy. $$
We have already verified the estimate for the $L^2$-norm, so we focus on the double integral.
$$ \int_\Omega\int_\Omega \frac{|(fg)(x) - (fg)(y)|^2}{|x-y|^{2s+N}}dxdy \lesssim $$
$$ \lesssim \int_\Omega\int_\Omega |f(x)|^2\frac{|g(x)-g(y)|^2}{|x-y|^{2s+N}}dxdy + \int_\Omega\int_\Omega |g(y)|^2\frac{|f(x)-f(y)|^2}{|x-y|^{2s+N}}dxdy. $$
For what concerns the first term, we can find $\frac{N}{2}<r<s+1$ and $\theta\in(0,1)$ as before:
$$ \int_\Omega\int_\Omega |f(x)|^2\frac{|g(x)-g(y)|^2}{|x-y|^{2s+N}}dxdy\le $$
$$ \le \|f\|_{L^\infty(\Omega)}^2\|g\|_{H^s(\Omega)}^2\lesssim \|f\|_{H^r(\Omega)}^2\|g\|_{H^s(\Omega)}^2\lesssim \|f\|_{H^s(\Omega)}^{2\theta}\|f\|_{H^{s+1}(\Omega)}^{2(1-\theta)}\|g\|_{H^s(\Omega)}^2. $$
The last term can be done as before with $g$ in place of $f$. Finally, the case $s\ge 1$ and $s\not\in\mathbb N$ can be got following the previous strategies.
\end{proof}
As it follows from the proof, we can ask $\theta>0$ because $s>\frac{N}{2}-1$. As a consequence of Lemma \ref{l.Leib.1}, we get a Leibniz formula which will be useful for the global existence.
\begin{lem}\label{l.Leib.2}
    Let $s> \frac{N}{2}-1$, let  $f,g\in H^{s+1}(\Omega)$ then it holds
    $$ \|fg\|_{H^s(\Omega)}\le C(\Omega,s)\left[ \|f\|_{H^{s+1}(\Omega)}\|g\|_{H^s(\Omega)} + \|g\|_{H^{s+1}(\Omega)}\|f\|_{H^s(\Omega)}\right]. $$
\end{lem}
For what concerns the trilinear estimates, we get similar results:
\begin{lem}\label{l.Leib.3}
    Let $s> \frac{N}{2}-1$ and $f,g,h\in H^{s+1}(\Omega)$, then there is $\theta=\theta(s)\in(0,1]$ such that
    $$ \|fg\|_{H^s(\Omega)}\le $$
    $$ \le C(\Omega,s) \|h\|_{H^{s+1}(\Omega)}\left[\|f\|_{H^s(\Omega)}^\theta\|f\|_{H^{s+1}(\Omega)}^{1-\theta}\|g\|_{H^s(\Omega)} + \|g\|_{H^s(\Omega)}^\theta\|g\|_{H^{s+1}(\Omega)}^{1-\theta}\|f\|_{H^s(\Omega)}\right], $$
    $$ \|fgh\|_{H^s(\Omega)}\le C(\Omega,s) \|h\|_{H^{s+1}(\Omega)}\left[\|f\|_{H^{s+1}(\Omega)}\|g\|_{H^s(\Omega)} + \|g\|_{H^{s+1}(\Omega)}\|f\|_{H^s(\Omega)}  \right]. $$
\end{lem}
\begin{proof}\hfill\\
The proof is similar to the previous one, so we consider just the case $s\in\mathbb N$ for simplicity. We notice that it is sufficient to prove that 
$$ \|fgh\|_{H^s(\Omega)}\lesssim \|h\|_{H^{s+1}(\Omega)}\|fg\|_{H^s(\Omega)}, $$
because the estimate will follow then by Lemma \ref{l.Leib.1} and \ref{l.Leib.2}. When $s=0$
$$ \|fgh\|_{L^2(\Omega)}\le \|h\|_{L^\infty(\Omega)}\|fg\|_{L^2(\Omega)}\lesssim \|h\|_{H^{s+1}(\Omega)}\|fg\|_{L^2(\Omega)}, $$
where we used that $L^\infty(\Omega)\hookrightarrow H^{s+1}(\Omega)$ for $s>\frac{N}{2}-1$. On the other hand, for any $\beta\in\mathbb N^N$ with $|\beta|=s$, we have
$$  \|D^\beta(fgh)\|_{L^2(\Omega)}\le \|D^\beta (fg)h\|_{L^2(\Omega)} + \|fg D^\beta h\|_{L^2(\Omega)} + \sum_{\gamma_1+\gamma_2=\beta,\:|\gamma_j|\ge 1} \|D^{\gamma_1}(fg)D^{\gamma_2}h\|_{L^2(\Omega)}. $$
The first term can be treated similarly: 
$$ \|D^\beta (fg)h\|_{L^2(\Omega)}\le \|h\|_{L^\infty(\Omega)}\|fg\|_{H^s(\Omega)}\lesssim  \|h\|_{H^{s+1}(\Omega)}\|fg\|_{H^s(\Omega)}. $$
For the second term we notice that $L^N(\Omega)\hookrightarrow H^s(\Omega)$ for $s\ge \frac{N}{2}-1$, so 
$$ \|fg D^\beta  h\|_{L^2(\Omega)}\le \|fg\|_{L^N(\Omega)}\|D^\beta h\|_{L^\frac{2N}{N-2}(\Omega)}\lesssim \|fg\|_{H^s(\Omega)}\|D^\beta h\|_{H^1(\Omega)}\lesssim  \|h\|_{H^{s+1}(\Omega)}\|fg\|_{H^s(\Omega)}. $$
For the last term, we take $p,q$ as in the previous proof so that 
$$ \sum_{\gamma_1+\gamma_2=\beta,\:|\gamma_j|\ge 1} \|D^{\gamma_1}(fg)D^{\gamma_2}h\|_{L^2(\Omega)}\le \|D^{\gamma_1}(fg)\|_{L^p(\Omega)}\|D^{\gamma_2}h\|_{L^q(\Omega)}\lesssim \|h\|_{H^{s+1}(\Omega)}\|fg\|_{H^s(\Omega)}. $$
\end{proof}

Now we can bound the nonlinear terms of \eqref{sys.loc-red.}:

\begin{lem}\label{l.bil.e.loc.}
Let $T>0$, $s>\frac{N}{2}-1$ and $z,w\in X^s_T(\Omega)$, then we can find $\gamma=\gamma(s)\in(0,1)$ such that

\begin{equation}\label{bil.e.1}
    \|z\nabla w\|_{L^1((0,T);L^1(\Omega))} \le C(N,s) T^{1/2} \|z\|_{X^s_T(\Omega)}\|w\|_{X^s_T(\Omega)},
\end{equation}
\begin{equation}\label{bil.e.2}
    \|z\nabla w\|_{L^1((0,T);L^2(\Omega))} \le C(N,s) T^{\gamma} \|z\|_{X^s_T(\Omega)}\|w\|_{X^s_T(\Omega)},
\end{equation}
\begin{equation}\label{bil.e.3}
    \|z \nabla w\|_{L^\frac{2}{2-s}((0,T);L^2(\Omega))} \le C(N,s)  T^\gamma \|z\|_{X^s_T(\Omega)}\|w\|_{X^s_T(\Omega)}\quad s\in[0,1],
    \end{equation}
\begin{equation}\label{bil.e.4}
    \|z w\|_{L^2((0,T);H^s(\Omega))} \le C(N,s)  T^\gamma \|z\|_{X^s_T(\Omega)}\|w\|_{X^s_T(\Omega)}.
    \end{equation}
\end{lem}
When $s=\frac{N}{2}-1$ the previous estimates are true with $\gamma=0$. We will prove the local existence by a contraction argument, therefore is important to have $\gamma>0$.
\begin{proof}\hfill\\
The first estimate is immediate:
$$ \|z\nabla w\|_{L^1((0,T);L^1(\Omega))}\le \left\|\|z(t)\|_{L^2(\Omega)}\|\nabla w(t)\|_{L^2(\Omega)}\right\|_{L^1((0,T))}\le $$
$$ \le \|z\|_{L^\infty((0,T);H^s(\Omega))}\|w\|_{L^1((0,T);H^{s+1}(\Omega))}\le T^{1/2}\|z\|_{X^s_T(\Omega)}\|w\|_{X^s_T(\Omega)}. $$
The second one is similar:
$$ \|z(t)\nabla w(t)\|_{L^2(\Omega)}\le \|z(t)\|_{L^\infty(\Omega)}\|w(t)\|_{H^1(\Omega)}\lesssim \|z(t)\|_{H^{s+1}(\Omega)}\|w(t)\|_{H^1(\Omega)}. $$
If $s\ge 1$ then
$$ \|z\nabla w\|_{L^1((0,T);L^2(\Omega))}\lesssim \left\|\|z(t)\|_{H^{s+1}(\Omega)}\|w(t)\|_{H^s(\Omega)}\right\|_{L^1((0,T))}\le $$
$$ \le \|w\|_{L^\infty((0,T);H^s(\Omega))}\|z\|_{L^1((0,T);H^{s+1}(\Omega))}\le T^\frac{1}{2}\|w\|_{X^s_T(\Omega)}\|z\|_{X^s_T(\Omega)}.  $$
Otherwise, if $s<1$, then 
$$ \|w(t)\|_{H^1(\Omega)}\lesssim \|w(t)\|_{H^s(\Omega)}^s\|w(t)\|_{H^{s+1}(\Omega)}^{1-s}, $$
so 
$$ \|z\nabla w\|_{L^1((0,T);L^2(\Omega))}\lesssim \|w\|_{L^\infty((0,T);H^s(\Omega))}^s\left\|\|z(t)\|_{H^{s+1}(\Omega)}\|w(t)\|_{H^{s+1}(\Omega)}^{1-s}\right\|_{L^1((0,T))}\le $$
$$ \le T^\frac{2-s}{2}\|z\|_{X^s_T(\Omega)}\|w\|_{X^s_T(\Omega)}. $$
For what concerns the third estimate \eqref{bil.e.3}
$$ \|z(t)\nabla w(t)\|_{L^2(\Omega)}\le \|z(t)\|_{L^\frac{2N}{N-2}(\Omega)}\|\nabla w(t)\|_{L^N(\Omega)}. $$
By Sobolev embedding and interpolation estimates (we recall $s\in[0,1]$)
$$ \|z(t)\|_{L^\frac{2N}{N-2}(\Omega)}\lesssim \|z(t)\|_{H^1(\Omega)}\lesssim \|z(t)\|_{H^s(\Omega)}^s\|z(t)\|_{H^{s+1}(\Omega)}^{1-s}. $$
On the other hand, since $s>\frac{N}{2}-1$, by Sobolev embedding and interpolation estimates we can find $\varepsilon>0$ such that 
$$ \|\nabla w(t)\|_{L^N(\Omega)}\lesssim \|w\|_{H^s(\Omega)}^\varepsilon \|w\|_{H^{s+1}(\Omega)}^{1-\varepsilon}. $$
Therefore
$$ \|z\nabla w\|_{L^\frac{2}{2-s}((0,T);L^2(\Omega))}\lesssim $$
$$ \lesssim \|z\|_{L^\infty((0,T);H^s(\Omega))}^s\|w\|_{L^\infty((0,T);H^s(\Omega))}^\varepsilon\left\|\|z(t)\|_{H^{s+1}(\Omega)}^{1-s}\|w(t)\|_{H^{s+1}(\Omega)}^{1-\varepsilon}\right\|_{L^\frac{2}{2-s}((0,T))} \le $$
$$ \le T^\frac{\varepsilon}{2}\|z\|_{L^\infty((0,T);H^s(\Omega))}^s\|w\|_{L^\infty((0,T);H^s(\Omega))}^\varepsilon \|z\|_{L^2((0,T);H^{s+1}(\Omega))}^{1-s}\|w\|_{L^2((0,T);H^{s+1}(\Omega))}^{1-\varepsilon}. $$
Finally, the estimate \eqref{bil.e.4} comes from Lemma \ref{l.Leib.1}:
$$ \|z(t)w(t)\|_{H^s(\Omega)}\lesssim $$
$$ \lesssim \|z(t)\|_{H^s(\Omega)}^\theta\|z(t)\|_{H^{s+1}(\Omega)}^{1-\theta}\|w(t)\|_{H^s(\Omega)} + \|w(t)\|_{H^s(\Omega)}^\theta\|w(t)\|_{H^{s+1}(\Omega)}^{1-\theta}\|z(t)\|_{H^s(\Omega)}, $$
for some $\theta\in(0,1]$. So
$$ \|z w\|_{L^2((0,T);H^s(\Omega))}\lesssim \|z\|_{L^\infty((0,T);H^s(\Omega))}^\theta\|w\|_{L^\infty((0,T);H^s(\Omega))}\|z\|_{L^{2(1-\theta)}((0,T);H^{s+1}(\Omega))}^{1-\theta} + $$
$$ + \|w\|_{L^\infty((0,T);H^s(\Omega))}^\theta\|z\|_{L^\infty((0,T);H^s(\Omega))}\|w\|_{L^{2(1-\theta)}((0,T);H^{s+1}(\Omega))}^{1-\theta}\le $$
$$ \le T^\theta \|z\|_{X^s_T(\Omega)}\|w\|_{X^s_T(\Omega)}. $$
\end{proof}
\begin{lem}\label{l.tril.e.loc.}
Let $s> \frac{N}{2}-1$, $T>0$, let $w,z\in X^s_T(\Omega)$ and $h\in X^{s+1}_T(\Omega)$, then
\begin{equation}\label{tril.e.1}
    \|zwh\|_{L^1((0,T);L^1(\Omega))}\le C(\Omega,s)T\|z\|_{X^s_T(\Omega)}\|w\|_{X^s_T(\Omega)}\|h\|_{X^{s+1}_T(\Omega)},
\end{equation}
\begin{equation}\label{tril.e.2}
        \|zwh\|_{L^2((0,T);H^{s}(\Omega))}\le C(\Omega,s) T^\gamma\|z\|_{X^s_T(\Omega)}\|w\|_{X^s_T(\Omega)}\|h\|_{X^{s+1}_T(\Omega)},
    \end{equation}
    for some $\gamma=\gamma(s)\in(0,1)$.
\end{lem}
\begin{proof}\hfill\\
Again, the first estimate is easy:
$$ \|zwh\|_{L^1((0,T);L^1(\Omega))}\le \left\|\|h(t)\|_{L^\infty(\Omega)}\|w(t)\|_{L^2(\Omega)}\|z(t)\|_{L^2(\Omega)}\right\|_{L^1((0,T))}\lesssim  $$
$$ \lesssim T\|h\|_{L^\infty((0,T);H^{s+1}(\Omega))}\|w\|_{L^\infty((0,T);H^s(\Omega)}\|z\|_{L^\infty((0,T);H^{s}(\Omega))}, $$
where we have used that $L^\infty(\Omega)\hookrightarrow H^{s+1}(\Omega)$ for $s>\frac{N}{2}-1$. For the second estimate it is sufficient to apply Lemma \ref{l.Leib.3} and proceed as in the previous lemma.
\end{proof}

\subsection{Proof of Theorems \ref{t.loc.ex.RN} and \ref{t.loc.ex.dom.}}

At the beginning of the Section \ref{sec.lin.es.}, we decided to study the Ericksen-Leslie system without the condition $|\eta+d|=1$. In fact, 
we are going to prove the property:
$$ |\eta+d_0|=1\:\Rightarrow\: |\eta+d(t)|=1\quad \text{for a.e.}\:\:t\in(0,T). $$
If it is true, it is sufficient to prove the local existence for the simplified system \eqref{sys.loc-red.} in order to get Theorems \ref{t.loc.ex.RN} and \ref{t.loc.ex.dom.}. As we mentioned above, we will follow the same argument of \cite{HNP16}.
\begin{lem}\label{l.un.loc.}
    Let $s> \frac{N}{2}-1$, $T>0$, $\eta\in \mathbb R^N$,
    $$ u_0\in H^s_{\mathbb P\Delta_D}\left(\Omega;\mathbb{R}^N\right),\quad d_0\in H^{s+1}_{\Delta_N}\left(\Omega;\mathbb R^N\right), $$
    with $|\eta+d_0|=1$, let $(u,d)\in X^s_T(\Omega)\times X^{s+1}_T(\Omega)$ be a solution for \eqref{sys.loc-red.}, then $|\eta+d(t)|=1$ for a.e. $t\in(0,T)$.
\end{lem}
\begin{proof}\hfill\\
    Let $\varphi\coloneqq |\eta+d|^2-1$, then
    $$ \partial_j\varphi=2(d+\eta)\cdot \partial_jd,\quad \Delta \varphi=2(d+\eta)\cdot\Delta d+2|\nabla d|^2. $$
    So
    $$ (\partial_t-\Delta)\varphi=2(d+\eta)\cdot(\partial_t-\Delta)d-2|\nabla d|^2= $$
    $$ = 2(d+\eta)\cdot\left[-u\cdot\nabla d+|\nabla d|^2(d+\eta)\right]-2|\nabla d|^2 = $$
    $$ = -2(d+\eta)\cdot(u\cdot \nabla d)+2|\nabla d|^2\varphi=-u\cdot\nabla \varphi+2|\nabla d|^2\varphi. $$
    In particular we notice that
    $$ \partial_\nu\varphi=2(d+\eta)\cdot \partial_\nu d, $$
    so $\varphi$ solves in the sense of distributions the system
    $$ \left\{\begin{array}{ll}
        (\partial_t-\Delta)\varphi=-u\cdot\nabla \varphi+2|\nabla d|^2\varphi & (0,T)\times \Omega \\
        \partial_\nu \varphi=0 & (0,T)\times\partial\Omega \\
        \varphi(0)=0 & \Omega.
    \end{array}\right. $$
    Then by Theorem \ref{t.lin.ex.RN} and \ref{t.lin.ex.dom.} (considering just the heat equation) we get that
    $$ \|\varphi\|_{X^{s+1}_T(\Omega)}\lesssim \|-u\cdot\nabla \varphi+2|\nabla d|^2\varphi\|_{L^1((0,T);L^1(\Omega))}+ \|-u\cdot\nabla \varphi+2|\nabla d|^2\varphi\|_{L^2((0,T);H^{s}(\Omega))}. $$
    From \eqref{bil.e.1} and \eqref{tril.e.1} we have that 
    $$  \|-u\cdot\nabla \varphi+2|\nabla d|^2\varphi\|_{L^1((0,T);L^1(\Omega))}\lesssim T^{\gamma}\left(\|u\|_{X^s_T(\Omega)}+\|d\|_{X^{s+1}_T(\Omega)}^2\right)\|\varphi\|_{X^{s+1}_T(\Omega)}, $$
    while, by the estimates \eqref{bil.e.4} and \eqref{tril.e.2}, we get 
    $$ \|-u\cdot\nabla \varphi+2|\nabla d|^2\varphi\|_{L^2((0,T);H^{s}(\Omega))}\le T^\gamma\left(\|u\|_{X^s_T(\Omega)}+\|d\|_{X^{s+1}_T(\Omega)}^2\right)\|\varphi\|_{X^{s+1}_T(\Omega)}, $$
    for some $\gamma>0$. In conclusion
    $$ \|\varphi\|_{X^{s+1}_T(\Omega)}\lesssim T^\gamma\left[\|u\|_{X^s_T(\Omega)}+\|d\|_{X^{s+1}_T(\Omega)}^2\right]\|\varphi\|_{X^{s+1}_T(\Omega)}. $$
    So, for $T_0\ll1$, we get that $\varphi=0$ for a.e. $t\in(0,T_0)$ and since the choice of $T_0$ depends just on global data, we can prove easily that $\varphi=0$ for a.e. $t\in(0,T)$.
    \end{proof}
We can now prove the local existence theorems. The proof of Theorem \ref{t.loc.ex.RN}, up to use estimate \eqref{bil.e.2} in place of \eqref{bil.e.1}, can be done analogously. 
\begin{proof}[Proof of Theorem \ref{t.loc.ex.dom.}]\hfill\\
Let us define the space
$$ Y_T\coloneqq \{(w,\theta)\in X^s_T(\Omega)\times X^{s+1}_T(\Omega)\mid w(t)\in J_2(\Omega)\:\:\text{for a.e.}\:\:t\in(0,T), \:\:\|(w,\theta)\|_{Y_T}<+\infty\},  $$
where
$$ \|(w,\theta)\|_{Y_T}\coloneqq\|w\|_{X^s_T(\Omega)}+\|\theta\|_{X^{s+1}_T(\Omega)}. $$
Now we define the map $\Phi\colon Y_T\to Y_T$ such that the function $\Phi(w,\theta)=(u,d)$ solves
\begin{equation}
    \left\{\begin{array}{ll}
    (\partial_t-\mathbb{P}\Delta)u=-\mathbb{P}(w\cdot \nabla w +{\rm Div}(\nabla \theta\odot\nabla \theta)) & (0,T)\times\Omega \\
    (\partial_t-\Delta)d=-w\cdot \nabla \theta+|\nabla \theta|^2(\eta+\theta) &(0,T)\times\Omega \\
    u=0 \quad \partial_\nu d=0 & (0,T)\times\partial\Omega \\
    u(0)=u_0,\quad d(0)=d_0 & \Omega,
    \end{array}\right.
\end{equation}
where $d_0\coloneqq v_0-\eta$. The plan of the proof is to prove that $\Phi$ is a contraction in a suitable subspace of $Y_T$. 

\textbf{Step 1}: As we did in the proof of Lemma \ref{l.un.loc.}, using Theorem \ref{t.lin.ex.dom.} and the estimates \eqref{bil.e.1}, \eqref{bil.e.3}, \eqref{bil.e.4}, \eqref{tril.e.1} and \eqref{tril.e.2}, we can find $\gamma>0$ such that
\begin{equation}\label{nl.loc.1}
    \begin{aligned}
 \|\Phi(w,\theta)\|_{Y_T}=\|u\|_{X^s_T(\Omega)}+\|d\|_{X^{s+1}_T(\Omega)} & \le C(T)\left[\|u_0\|_{H^s(\Omega)}+\|d_0\|_{H^{s+1}(\Omega)}\right] + \\
 & + T^\gamma\left(\|(w,\theta)\|_{Y_T}^2+\|(w,\theta)\|_{Y_T}^3\right),
    \end{aligned}
\end{equation}
where we defined $\Phi$ at the beginning of the proof. Similarly, for any $(w_1,\theta_1),(w_2,\theta_2)\in Y_T$ it holds 
\begin{equation}\label{nl.loc.2}
    \begin{aligned}
        & \|\Phi(w_1,\theta_1)-\Phi(w_2,\theta_2)\|_{Y_T}\lesssim  \\
        \lesssim T^\gamma\left(\|(w_1,\theta_1)\|_{Y_T}+\|(w_2,\theta_2)\|_{Y_T}\right. & \left.+ \|(w_1,\theta_1)\|_{Y_T}^2+\|(w_2,\theta_2)\|_{Y_T}^2\right)\|(w_1,\theta_1)-(w_2,\theta_2)\|_{Y_T}.
    \end{aligned}
\end{equation}

\textbf{Step 2}: Let us define now the space
$$ Z_\omega\coloneqq \{(w,\theta)\in Y_T\mid w(0)=u_0,\quad \theta(0)=d_0,\quad \|(w,\theta)\|_{Y_T}\le \omega\}, $$
for $\omega>0$ to be defined. We want to prove $\Phi\colon Z_\omega\to Z_\omega$ is a contraction for a suitable choice of $\omega$: thanks to \eqref{nl.loc.1} we know that exists $C>0$ such that
$$ \|\Phi(w,\theta)\|_{Y_T}\le C(T)\left[\|u_0\|_{H^s(\Omega)}+\|d_0\|_{H^{s+1}(\Omega)}\right] + CT^\gamma(\omega^2+\omega^3)\le $$
$$ \le C\left[RC(T)+T^\gamma(\omega^2+\omega^3)\right]. $$
Looking at the proofs of Lemma \ref{l.bil.e.loc.} and \ref{l.tril.e.loc.}, it can be seen that $C(T)\le K$. So, if we choose $\omega$ such that
$$ KCR\le \frac{\omega}{2}, $$
and $T\in(0,1)$ sufficiently small, then we have $\Phi(w,\theta)\in Z_\omega$. On the other hand, thanks to \eqref{nl.loc.2} there is $M>0$ such that, for any $(w_1,\theta_1),(w_2,\theta_2)\in Z_\omega$, it holds 
$$ \|\Phi(w_1,\theta_1)-\Phi(w_2,\theta_2)\|_{Y_T}\le 2M(\omega+\omega^2) T^\gamma \|(w_1,\theta_1)-(w_2,\theta_2)\|_{Y_T}. $$
Therefore, choosing $T>0$ sufficiently small, we get that $\Phi\colon Z_\omega\to Z_\omega$ is a contraction and we get a solution for the system \eqref{sys.loc-red.} in $(0,T)$.

\vspace{2mm}

\textbf{Step 3:} The solution we found is unique in $X^s_T(\Omega)\times X^{s+1}_T(\Omega)$: let $(u_1,d_1), (u_2,d_2)\in Y_T$ and let 
$$ R_1=\|(u_1,d_1)\|_{Y_T},\quad R_2=\|(u_2,d_2)\|_{Y_T}. $$
In particular they are fixed points for the function $\Phi$. Then, for any $T_0<T$, by the estimate \eqref{nl.loc.2} we get
$$ \|(u_1,d_1)-(u_2,d_2)\|_{Y_{T_0}}\lesssim T_0^\gamma \left(R_1+R_2+R_1^2+R_2^2\right)\|(u_1,d_1)-(u_2,d_2)\|_{Y_{T_0}}. $$
So, if we choose $T_0=T_0(R_1,R_2)$ sufficiently small, we get that $(u_1,d_1)=(u_2,d_2)$ in $Y_{T_0}$. In particular
$$ u_1(t,x)=u_2(t,x),\quad d_1(t,x)=d_2(t,x)\quad\forall t\in(0,T_0)\:\: \text{for a.e.}\:\:x\in \Omega. $$
Since the choice of $T_0$ does not depend of $u_0$ and $d_0$, we can repeat the argument choosing as starting point $T_0$. In this way, in a finite number of steps, we prove that $(u_1,d_1)=(u_2,d_2)$ in $Y_T$.

\vspace{2mm}

\textbf{Step 4:} Finally, we consider $v(t,x)=d(t,x)+\eta$. Thanks to Lemma \ref{l.un.loc.}, the couple $(u,v)$ solves the Ericksen-Leslie system \eqref{EL.sys-proj.} in $(0,T)$. The existence of $p$ follows again by Theorem \ref{t.lin.ex.dom.}.
\end{proof}

\section{Global Existence}\label{sec.gl.ex.}

We focus on the reduced global system
\begin{equation}\label{sys.gl-red.}
    \left\{\begin{array}{ll}
       (\partial_t-\Delta)u= - \mathbb P\left(u\cdot \nabla u + {\rm Div}\left(\nabla d\odot\nabla d\right)\right)  & \mathbb R_+\times \Omega \\
       {\rm div}u=0 & \mathbb R_+\times\Omega \\
       (\partial_t-\Delta)d +u\cdot \nabla d = |\nabla d|^2(d+\eta) & \mathbb R_+\times\Omega \\
       \partial_\nu d=0,\quad u=0 & \mathbb R_+\times \partial\Omega \\
       u(0)=u_0,\quad d(0)=d_0 & \Omega.
    \end{array}\right.     
\end{equation}
As for the local case, the first task is to prove that the nonlinearities of \eqref{sys.gl-red.} are bounded in $X^s(\Omega)\times X^{s+1}_T(\Omega)$. In the second part we will introduce a subspace of $X^s(\Omega)\times X^{s+1}(\Omega)$ for the contraction argument. This space will allow us to control the time decay of the solution.

\subsection{Multilinear Global Estimates}

Thanks to the Leibniz rules proved in Lemma \ref{l.Leib.2} and \ref{l.Leib.3} we can bound the bilinear and the trilinear terms from the nonlinearities of \eqref{sys.gl-red.}: 
\begin{lem}
Let $s> \frac{N}{2}-1$ and $z,w\in X^{s}(\Omega)$, then 
\begin{equation}\label{bil.e.5}
    \|z\nabla w\|_{L^1(\mathbb R_+;L^1(\Omega))} \le C(\Omega,s)\|z\|_{X^{s}}\|w\|_{X^s(\Omega)},
\end{equation}
\begin{equation}\label{bil.e.6}
    \|z\nabla w\|_{L^1(\mathbb R_+;L^2(\Omega))} \le C(\Omega,s)\|z\|_{X^s(\Omega)}\|w\|_{X^s(\Omega)},
\end{equation}
\begin{equation}\label{bil.e.7}
    \|z\nabla w\|_{L^\frac{2}{2-s}(\mathbb R_+;L^2(\Omega))} \le C(\Omega,s) \|z\|_{X^s(\Omega)}\|w\|_{X^s(\Omega)}\quad s\in[0,1],
\end{equation}
\begin{equation}\label{bil.e.8}
    \|z w\|_{L^2(\mathbb R_+;H^s(\Omega))} \le C(\Omega,s) \|z\|_{X^s(\Omega)}\|w\|_{X^s(\Omega)}.
    \end{equation}
\end{lem}
\begin{proof}\hfill\\
The first and second estimates are similar to \eqref{bil.e.1} and \eqref{bil.e.2}:
$$ \|z\nabla w\|_{L^1(\mathbb R_+;L^1(\Omega))}\le \left\|\|z(t)\|_{L^2(\Omega)}\|\nabla w(t)\|_{L^2(\Omega)}\right\|_{L^1(\mathbb R_+)}\le \|z\|_{L^2(\mathbb R_+;L^2(\Omega))}\|w\|_{L^2(\mathbb R_+;H^1(\Omega))}, $$
$$ \|z\nabla w\|_{L^1(\mathbb R_+;L^2(\Omega))}\lesssim \left\|\|z(t)\|_{H^1(\Omega)}\|w\|_{H^{s+1}(\Omega)}\right\|_{L^1(\mathbb R_+)}\le \|z\|_{L^2(\mathbb R_+;H^{s+1}(\Omega))}\|w\|_{L^2(\mathbb R_+;H^{s+1}(\Omega))}. $$
For what concerns the third estimate \eqref{bil.e.7}, we have that
$$ \|z(t)\nabla w(t)\|_{L^2(\Omega)}\lesssim \|z(t)\|_{L^\frac{2N}{N-2}(\Omega)}\|\nabla w(t)\|_{L^N(\Omega)}\lesssim \|z(t)\|_{H^1(\Omega)}\|\nabla w(t)\|_{H^s(\Omega)}. $$
By interpolation
$$ \|z(t)\|_{H^1(\Omega)}\lesssim \|z(t)\|_{H^s(\Omega)}^s\|z(t)\|_{H^{s+1}(\Omega)}^{1-s}, $$
provided $s\ge \frac{N}{2}-1$. Therefore
$$ \|z\nabla w\|_{L^\frac{2}{2-s}(\mathbb R_+;L^2(\Omega))}\lesssim \|z\|_{L^\infty(\mathbb R_+;H^s(\Omega))}^s \left\|\|z(t)\|_{H^{s+1}(\Omega)}^{1-s}\|\nabla w(t)\|_{H^s(\Omega)}\right\|_{L^\frac{2}{2-s}(\mathbb R_+)}\lesssim $$
$$ \lesssim \|z\|_{L^\infty(\mathbb R_+;H^s(\Omega))}^s \|z\|_{L^2(\mathbb R+;H^{s+1}(\Omega))}^{1-s}\|w\|_{L^2(\mathbb R_+;H^{s+1}(\Omega))}. $$
Finally the estimate \eqref{bil.e.8} follows from Lemma \ref{l.Leib.2}:
$$ \|z(t) w(t)\|_{H^s(\Omega)}\lesssim \|z(t)\|_{H^{s+1}(\Omega)}\|w(t)\|_{H^s(\Omega)} + \|w(t)\|_{H^{s+1}(\Omega)}\|z(t)\|_{H^s(\Omega)}. $$
So
$$ \|zw\|_{L^2(\mathbb R_+;H^s(\Omega))}\lesssim $$
$$ \lesssim \|z\|_{L^2(\mathbb R_+;H^{s+1}(\Omega))}\|w\|_{L^\infty(\mathbb R_+;H^{s}(\Omega))} + \|w\|_{L^2(\mathbb R_+;H^{s+1}(\Omega))}\|z\|_{L^\infty(\mathbb R_+;H^{s}(\Omega))}. $$
\end{proof}
The bilinear estimates we just proved are true also for $s\ge \frac{N}{2}-1$ as it can be seen from the proof (we only ask $L^N(\Omega)\hookrightarrow H^s(\Omega)$). However, the trilinear estimate we are going to prove, needs the condition $s>\frac{N}{2}-1$:
\begin{lem}
Let $s> \frac{N}{2}-1$, let $w,z\in X^s(\Omega)$, $h\in X^{s+1}(\Omega)$, then
\begin{equation}\label{tril.e.3}
    \|zwh\|_{L^1(\mathbb R_+;L^1(\Omega))}\le C(\Omega,s) \|z\|_{X^s(\Omega)}\|w\|_{X^s(\Omega)}\|h\|_{X^{s+1}(\Omega)},
\end{equation}
\begin{equation}\label{tril.e.4}
        \|zwh\|_{L^2(\mathbb R_+;H^{s}(\Omega))}\le C(\Omega,s)\|z\|_{X^s(\Omega)}\|w\|_{X^s(\Omega)}\|h\|_{X^{s+1}(\Omega)}.
    \end{equation}
\end{lem}
\begin{proof}\hfill\\
Again, the first estimate is easy
$$ \|zwh\|_{L^1(\mathbb R_+;L^1(\Omega))}\le \left\|\|h(t)\|_{L^\infty(\Omega)}\|w(t)\|_{L^2(\Omega)}\|z(t)\|_{L^2(\Omega)}\right\|_{L^1(\mathbb R_+)}\lesssim $$
$$ \lesssim \|h\|_{L^\infty(\mathbb R_+;H^{s+1}(\Omega))}\|w\|_{L^2(\mathbb R_+;L^2(\Omega)}\|z\|_{L^2(\mathbb R_+;L^2(\Omega))}. $$
For the second estimate we use Lemma \ref{l.Leib.3}:
$$ \|z(t)w(t)h(t)\|_{H^{s}(\Omega)}\lesssim   $$
$$ \left[\|z(t)\|_{H^{s+1}(\Omega)}\|w(t)\|_{H^{s}(\Omega)}+ \|z(t)\|_{H^{s}(\Omega)}\|w(t)\|_{H^{s+1}(\Omega)}\right]\|h(t)\|_{H^{s+1}(\Omega)}. $$
So we get  
$$  \|zwh\|_{L^2(\mathbb R_+;H^{s}(\Omega))}\lesssim $$
$$ \lesssim \left[\|z\|_{L^2(\mathbb R_+;H^{s+1}(\Omega))}\|w\|_{L^\infty(\mathbb R_+;H^{s}(\Omega)} + \|w\|_{L^2(\mathbb R_+;H^{s+1}(\Omega))}\|z\|_{L^\infty(\mathbb R_+;H^{s}(\Omega)}\right] \|h\|_{L^\infty(\mathbb R_+;H^{s+1}(\Omega))}. $$
\end{proof}

\subsection{Multilinear Decay Estimates in $\mathbb R^N$ and $\mathbb R^N_+$}

As it can be seen from the statements of the global existence theorems, the order of decay in times changes with the choice of $\Omega$. In fact, in order to control the decay of derivatives of the semi group flow we have to treat different cases of boundary conditions on the boundary of the domain. This implies that the commutators of derivatives and the semigroup are nonzero so we have difficulties in establishing dispersive estimates for derivatives of the semigroup flow. As first preliminary step, we consider the simpler case, when $\Omega=\mathbb R^N$ or $\mathbb R^N_+$. 

As we mentioned at the beginning of the section, we consider a subset of $X^s(\Omega)$ where the functions have a time decay:
\begin{defn}\hfill\\
Let $N\ge 3$, $s\ge 0$ and $k\in\mathbb N$, then we defined 
$$ X^{s}_k(\Omega)\coloneqq \{w\in X^s(\Omega)\mid \|w\|_{X^{s}_k(\Omega)}<+\infty\}, $$
with 
$$ \|w\|_{X^{s}_k(\Omega)}\coloneqq \|w\|_{X^s(\Omega)}+\sup_{t\le 2}t^{\alpha_{s,k}}\|w\|_{W^{k,\infty}(\Omega)} + \sum_{j=0}^{k+1}\sup_{t\ge 1}t^{\frac{N}{2}+\frac{j}{2}}\|\nabla^jw(t)\|_{L^\infty(\Omega)},
$$
where 
$$ \alpha_{s,k}=\left\{\begin{array}{ll}
    \frac{N-2(s-k)}{4} & s-k\le\frac{N}{2} \\
    0 & s-k>\frac{N}{2},
\end{array}\right. $$
with $X^{s}(\Omega)$ is defined in \eqref{def.X}.
\end{defn}
Even if we are interested on the time decay, we need to control the order of the singularity for $t$ near to 0. This will be clear in the proof.
\begin{rem}
About the choice of $\alpha_{s,k}$, we notice that $\alpha\in\left[0,\frac{1}{2}\right)$ when $s-k>\frac{N}{2}-1$.
\end{rem}
Before showing that the semi group flows for the heat and the Stokes equation are in $X^s_k(\Omega)$ we need a technical lemma:
\begin{lem}\label{l.tec.int.1}
Let $0\le\alpha_1,\alpha_2<1$, then it holds
$$ \int_0^t (t-\tau)^{-\alpha_1}\tau^{-\alpha_2}d\tau\le C(\alpha_1,\alpha_2) t^{1-\alpha_1-\alpha_2} \quad \forall t>0. $$
\end{lem}
\begin{proof}\hfill\\
With the change of variable $s=\frac{\tau}{t}$
$$ \int_0^t (t-\tau)^{-\alpha_1}\tau^{-\alpha_2}d\tau=t^{1-\alpha_1-\alpha_2}\int_0^1(1-s)^{-\alpha_1}s^{-\alpha_2}ds. $$
In order to conclude it is sufficient to notice that
$$ \int_0^1(1-s)^{-\alpha_1}s^{-\alpha_2}ds=\int_0^{1/2}(1-s)^{-\alpha_1}s^{-\alpha_2}ds+\int_{1/2}^1(1-s)^{-\alpha_1}s^{-\alpha_2}ds\lesssim  $$
$$ \lesssim \int_0^{1/2}s^{-\alpha_2}ds+\int_{1/2}^1(1-s)^{-\alpha_1} ds\simeq 1. $$
\end{proof}
As for the linear estimates, we start from the homogeneous term:
\begin{lem}\label{l.dec.1.RN}
    Let $k\in\mathbb N$ and $s>0$ such that $s-k>\frac{N}{2}-1$, let $u_0\in H^s\cap L^1(\mathbb R^N)$ and $d_0\in H^{s+1}\cap L^1(\mathbb R^N)$, then
    $$ \left\|e^{\mathbb P\Delta t}u_0\right\|_{X^s_k(\mathbb R^N)} \le C(N,s,k) \|u_0\|_{H^s\cap L^1(\mathbb R^N)}, $$
    $$ \left\|e^{\Delta t}d_0\right\|_{X^{s+1}_{k+1}(\mathbb R^N)} \le C(N,s,k) \|d_0\|_{H^{s+1}\cap L^1(\mathbb R^N)}. $$
\end{lem}
\begin{proof}\hfill\\
The estimates in $X^s(\mathbb R^N)$ and $X^{s+1}(\mathbb R^N)$ follows from Theorem \ref{t.lin.ex.RN}. For the $L^\infty(\mathbb R^N)$ estimates, let us suppose $s-k<\frac{N}{2}$ (otherwise the proof is easier). Let $t\le 2$ and $j=1,\ldots, k$, then we have 
$$ \|\nabla^je^{\mathbb P\Delta t}u_0\|_{L^\infty(\mathbb R^N)}=\|e^{\mathbb P\Delta t}\nabla^j u_0\|_{L^\infty(\mathbb R^N)}\lesssim t^{-\alpha_{s,k}}\|\nabla^ju_0\|_{L^\frac{2N}{N-2(s-k)}(\mathbb R^N)}\lesssim t^{-\alpha_{s,k}}\|u_0\|_{H^s(\mathbb R^N)}. $$
When $t\ge 1$
$$ \|\nabla^{j+1}e^{\mathbb P\Delta t}u_0\|_{L^\infty(\mathbb R^N)}\lesssim t^{-\frac{j+1}{2}-\frac{N}{2}}\|u_0\|_{L^1(\mathbb R^N)}. $$
Similar argument for $e^{\Delta t}d_0$.
\end{proof}

Let us pass to the inhomogeneous term:
\begin{lem}\label{l.dec.2.RN}

    Let $k\in\mathbb N$ and $s>0$ such that $s-k>\frac{N}{2}-1$, let $z,w\in X^s_k(\mathbb R^N)$ and $\theta,h\in X^{s+1}_{k+1}(\mathbb R^N)$, then
    \begin{equation}
    \left\|\int_0^te^{\mathbb P\Delta(t-\tau)}\mathbb Pw(\tau)\nabla z(\tau)d\tau\right\|_{X^s_k(\mathbb R^N)}\le C(N,s,k)\|z\|_{X^s_{k}(\mathbb R^N)}\|w\|_{X^s_{k}(\mathbb R^N)}.       
    \end{equation}
        \begin{equation}
    \left\|\int_0^te^{\mathbb P\Delta(t-\tau)}\mathbb P\nabla \theta(\tau)\nabla^2h(\tau)d\tau\right\|_{X^s_k(\mathbb R^N)}\le C(N,s,k)\|\theta\|_{X^{s+1}_{k+1}(\mathbb R^N)}\|\theta\|_{X^{s+1}_{k+1}(\mathbb R^N)}.       
    \end{equation}
\end{lem}
\begin{proof}\hfill\\
The function
$$ u(t)=\int_0^t e^{\mathbb P\Delta(t-\tau)}\mathbb P\left(w(\tau)\nabla z (\tau) + \nabla \theta(\tau) \nabla^2 h(\tau)\right)d\tau, $$
solves the system
$$ \left\{ \begin{array}{ll}
    (\partial_t-\mathbb P\Delta)u=\mathbb P\left( w\nabla z + \nabla \theta \nabla^2 h\right) & \mathbb R_+\times\mathbb R^N \\
    {\rm div}u=0 & \mathbb R_+\times\mathbb R^N \\
    u(0)=0 & \mathbb R^N. 
\end{array}\right. $$
Thanks to Theorem \ref{t.lin.ex.RN} (considering only the Stokes equation) and using the estimates \eqref{bil.e.6} and \eqref{bil.e.8}, we get that 
$$ \|u\|_{X^{s}(\mathbb R^N)} \lesssim \|w\|_{X^s(\mathbb R^N)}\|z\|_{X^{s}(\mathbb R^N)} + \|\theta\|_{X^{s+1}(\mathbb R^N)}\|h\|_{X^{s+1}(\mathbb R^N)}. $$
When $s-k>\frac{N}{2}$, the previous estimate concludes the proof by Sobolev embedding since
$$ W^{k,\infty}\left(\mathbb R^N\right)\hookrightarrow H^s\left(\mathbb R^N\right). $$
Let us consider then $s-k\le \frac{N}{2}$. Moreover, since $\nabla \theta,\nabla h\in X^s_k(\mathbb R^N)$, we can focus just on the term $w\nabla z$. Let us start with the case $k=0$. If $t\le 2$
$$ \left\|\int_0^te^{\mathbb P\Delta(t-\tau)}\mathbb P(w(\tau)\nabla z(\tau))d\tau\right\|_{L^\infty(\mathbb R^N)}\lesssim \int_0^t (t-\tau)^{-\frac{N}{2r}}\|w(\tau)\nabla z(\tau)\|_{L^r(\mathbb R^N)}d\tau, $$
for some $r\in(1,\infty)$. If we take $r=\frac{2N}{N-2s}$ and we use Sobolev embedding
$$ \int_0^t (t-\tau)^{-\frac{N}{2r}}\|w(\tau)\nabla z(\tau)\|_{L^r(\mathbb R^N)}d\tau\lesssim \|w\|_{X^s_0(\mathbb R^N)}\int_0^t (t-\tau)^{-\frac{N-2s}{4}}\tau^{-\alpha_{s,k}}\|\nabla z(\tau)\|_{H^s(\mathbb R^N)}d\tau\le $$
$$ \le \|w\|_{X^s_0(\mathbb R^N)}\|z\|_{X^s_0(\mathbb R^N)} \left(\int_0^t (t-\tau)^{-\frac{N-2s}{2}}\tau^{-2\alpha_{s,k}}d\tau\right)^\frac{1}{2}. $$
Since
$$ \frac{N-2s}{2}=2\alpha_{s,k}<1, $$
we can apply Lemma \ref{l.tec.int.1} so that
$$ \left(\int_0^t (t-\tau)^{-\frac{N-2s}{2}}\tau^{-2\alpha_{s,k}}d\tau\right)^\frac{1}{2}\simeq t^{\frac{1-N+2s}{2}}\lesssim t^{-\alpha_{s,k}}, $$
for $t\le 2$ and $s>\frac{N}{2}-1$. Otherwise, when $t\ge 2$ we split the integral
$$ \int_0^t e^{\mathbb P\Delta(t-\tau)}\mathbb Pw(\tau)\nabla z(\tau) d\tau= I_1 + I_2,  $$
where 
$$ I_1=\int_0^{t/2} e^{\mathbb P\Delta(t-\tau)}\mathbb Pw(\tau)\nabla z(\tau) d\tau, $$
$$ I_2= \int_{t/2}^t e^{\mathbb P\Delta(t-\tau)}\mathbb Pw(\tau)\nabla z(\tau) d\tau. $$
Then
$$ \|I_1\|_{L^\infty(\mathbb R^N)}\lesssim \int_0^{t/2}(t-\tau)^{-\frac{N}{2}}\|w(\tau)\nabla z(\tau)\|_{L^1(\mathbb R^N)}d\tau\lesssim $$
$$ \lesssim t^{-\frac{N}{2}}\int_0^{t/2} \|w(\tau)\|_{L^2(\mathbb R^N)}\|\nabla z(\tau)\|_{L^2(\mathbb R^N)} d\tau \le $$
$$ \le t^{-\frac{N}{2}}\|w\|_{L^2(\mathbb R_+;L^2(\mathbb R^N))}\|\nabla z\|_{L^2(\mathbb R_+;L^2(\mathbb R^N))}\le t^{-\frac{N}{2}}\|w\|_{X^s_0(\mathbb R^N)}\|z\|_{X^s_0(\mathbb R^N)}. $$
On the other hand
$$ \|I_2\|_{L^\infty(\mathbb R^N)} \lesssim \int_{t/2}^t (t-\tau)^{-\frac{N}{2r}}\|w(\tau)\nabla z(\tau)\|_{L^r(\mathbb R^N)}d\tau, $$
for some $r\in[1,\infty)$. We notice that
$$ \|w(\tau)\nabla z(\tau)\|_{L^r(\mathbb R^N)}\le \tau^{-\frac{N}{2}-\frac{1}{2}}\|z\|_{X^s_0(\mathbb R^N)}\|w(\tau)\|_{L^r(\mathbb R^N)}\le \tau^{-N-\frac{1}{2}+\frac{N}{r}}\|z\|_{X^s_0(\mathbb R^N)}\|w\|_{X^s_0(\mathbb R^N)}. $$
Therefore, for $r>\frac{N}{2}$, it holds
$$ \|I_2\|_{L^\infty(\mathbb R^N)}\lesssim  t^{-N-\frac{1}{2}+\frac{N}{r}}\|z\|_{X^s_0(\mathbb R^N)}\|w\|_{X^s_0(\mathbb R^N)}\int_{t/2}^t (t-\tau)^{-\frac{N}{2r}}d\tau \le  $$
$$ \le t^{-N+\frac{1}{2}+\frac{N}{2r}}\|w\|_{X^s_1(\mathbb R^N)}\|z\|_{X^s_1(\mathbb R^N)}. $$
Since for $N\ge 3$
$$ N-\frac{1}{2}>\frac{N}{2}, $$
then we can find $r$ sufficiently large such that 
$$ \|I_2\|_{L^\infty(\mathbb R^N)}\lesssim t^{-\frac{N}{2}}\|z\|_{X^s_0(\mathbb R^N)}\|w\|_{X^s_0(\mathbb R^N)}. $$
With the same approach it can be seen that 
$$ \|\nabla I_1\|_{L^\infty(\mathbb R^N)} \lesssim \int_0^{t/2}(t-\tau)^{-\frac{N}{2}-\frac{1}{2}}\|z(\tau)\nabla w(\tau)\|_{L^1(\mathbb R^N)}d\tau\lesssim t^{-\frac{N}{2}-\frac{1}{2}}\|z\|_{X^s_0(\mathbb R^N)}\|w\|_{X^s_0(\mathbb R^N)}, $$
$$ \|\nabla I_2\|_{L^\infty(\mathbb R^N)}\lesssim t^{-N+\frac{N}{2r}}\|z\|_{X^s_0(\mathbb R^N)}\|w\|_{X^s_0(\mathbb R^N)}\lesssim t^{-\frac{N}{2}-\frac{1}{2}}\|z\|_{X^s_0(\mathbb R^N)}\|w\|_{X^s_0(\mathbb R^N)}\quad r\gg1. $$
This prove the case $k=0$. Let now $k\ge 1$ and let $\beta\in\mathbb N^N$ with $|\beta|=k$. If $t\le 2$ we can find $\ell=1,\ldots, N$ and $\widetilde \beta\in\mathbb N^N$ such that $\beta=e_\ell+\widetilde \beta$. So for any $r\in(1,\infty)$
$$ \left\|D^\beta\int_0^t e^{\mathbb P\Delta(t-\tau)}\mathbb Pw(\tau)\nabla z(\tau)d\tau\right\|_{L^\infty(\mathbb R^N)}\lesssim \int_0^t (t-\tau)^{-\frac{1}{2}-\frac{N}{2r}}\left\|D^{\widetilde \beta}(w(\tau)\nabla z(\tau))\right\|_{L^r(\mathbb R^N)}d\tau. $$
On the other hand, for any $r\ge 2$
$$ \left\|D^{\widetilde \beta}(w(\tau)\nabla z(\tau))\right\|_{L^r(\mathbb R^N)}\lesssim \sum_{\gamma_1+\gamma_2=\widetilde \beta,\:\:|\gamma_j|\ge 1}\left\|D^{\gamma_1}w(\tau)D^{\gamma_2}\nabla z(\tau)\right\|_{L^r(\mathbb R^N)} \le $$
$$ \le \sum_{\gamma_1+\gamma_2=\widetilde \beta,\:\:|\gamma_j|\ge 1} \left\|D^{\gamma_1}\nabla z(\tau)\right\|_{L^\infty(\mathbb R^N)}\left\|D^{\gamma_2}w(\tau)\right\|_{L^\infty(\mathbb R^N)}^{1-\frac{2}{r}}\left\|D^{\gamma_2}w(\tau)\right\|_{L^2(\mathbb R^N)}^\frac{2}{r}. $$
Since $|\gamma_1|,|\gamma_2|\le k-1$, we have that 
$$ \left\|D^{\gamma_1}\nabla z(\tau)\right\|_{L^\infty(\mathbb R^N)}\le \tau^{-\alpha_{s,k}}\|z\|_{X^s_{k}(\mathbb R^N)}.\quad \left\|D^{\gamma_2}w(\tau)\right\|_{L^\infty(\mathbb R^N)}\le \tau^{-\alpha_{s,k}}\|w\|_{X^s_{k}(\mathbb R^N)}. $$
In particular, since $k-1\le s$, we have
$$ \left\|D^{\gamma_1}\nabla z(\tau)\right\|_{L^\infty(\mathbb R^N)}\left\|D^{\gamma_2}w(\tau)\right\|_{L^\infty(\mathbb R^N)}^{1-\frac{2}{r}}\left\|D^{\gamma_2}w(\tau)\right\|_{L^2(\mathbb R^N)}^\frac{2}{r}\le \tau^{-2\alpha_{s,k}+\frac{2\alpha_{s,k}}{r}}\|z\|_{X^s_{k}(\mathbb R^N)}\|w\|_{X^s_{k}(\mathbb R^N)}. $$
Therefore, since $2\alpha_{s,k}<1$, choosing $r$ sufficiently large we get by Lemma \ref{l.tec.int.1} that
$$ \int_0^t (t-\tau)^{-\frac{1}{2}-\frac{N}{2r}}\left\|D^{\widetilde \beta}(z(\tau)w(\tau))\right\|_{L^r(\mathbb R^N)}d\tau\lesssim  $$
$$ \lesssim \|z\|_{X^s_k(\mathbb R^N)}\|w\|_{X^s_{k}(\mathbb R^N)}\int_0^t (t-\tau)^{-\frac{1}{2}-\frac{N}{2r}}\tau^{-2\alpha_{s,k}+\frac{2\alpha_{s,k}}{r}}d\tau \lesssim $$
$$ \lesssim t^{\frac{1}{2}-2\alpha_{s,k}-\frac{N}{2r}+\frac{2\alpha_{s,k}}{r}}\|z\|_{X^s_k(\mathbb R^N)}\|w\|_{X^s_{k}(\mathbb R^N)}\lesssim t^{-\alpha_{s,k}}\|z\|_{X^s_k(\mathbb R^N)}\|w\|_{X^s_k(\mathbb R^N)}\quad t\le 2, $$
where in the last inequality we used that $\alpha_{s,k}<\frac{1}{2}$ when $s>\frac{N}{2}-1$.

Let us consider now the case $t\ge 2$. 
$$ D^\beta\int_0^t e^{\mathbb P\Delta(t-\tau)}\mathbb Pw(\tau)\nabla z(\tau)d\tau= I_1+I_2, $$
with 
$$ I_1\coloneqq D^\beta\int_0^{t/2} e^{\mathbb P\Delta(t-\tau)}\mathbb Pw(\tau)\nabla z(\tau)d\tau, $$
$$ I_2\coloneqq D^\beta\int_{t/2}^t e^{\mathbb P\Delta(t-\tau)}\mathbb Pw(\tau)\nabla z(\tau)d\tau. $$
So
$$ \|I_1\|_{L^\infty(\mathbb R^N)}\lesssim \int_0^{t/2}(t-\tau)^{-\frac{k}{2}-\frac{N}{2}}\|w(\tau)\nabla z(\tau)\|_{L^1(\mathbb R^N)}d\tau\le t^{-\frac{k}{2}-\frac{N}{2}} \|z\|_{X^s_{k}(\mathbb R^N)}\|w\|_{X^s_{k}(\mathbb R^N)}.  $$
On the other hand, as before we can find $\widetilde\beta$ and $\ell=1,\ldots, N$ such that $\beta=\widetilde\beta + e_\ell$ and 
$$ \|I_2\|_{L^\infty(\mathbb R^N)}\lesssim \int_{t/2}^t (t-\tau)^{-\frac{1}{2}-\frac{N}{2r}}\left\|D^{\widetilde \beta}(w(\tau)\nabla z(\tau))\right\|_{L^r(\mathbb R^N)}d\tau. $$
With the same strategy as before, we get 
$$ \left\|D^{\widetilde \beta}(z(\tau)w(\tau))\right\|_{L^r(\mathbb R^N)}\lesssim \|z\|_{X^s_{k}(\mathbb R^N)}\|w\|_{X^s_{k}(\mathbb R^N)}\sum_{\gamma_1+\gamma_2=\widetilde \beta}t^{-\left(\frac{N}{2}+\frac{|\gamma_1|+1}{2}\right)-\left(\frac{N}{2}+\frac{|\gamma_2|}{2}\right)\left(1-\frac{2}{r}\right)}. $$
Now we notice that 
$$ \frac{N}{2}+\frac{|\gamma_1|+1}{2}+\left(\frac{N}{2}+\frac{|\gamma_2|}{2}\right)\left(1-\frac{2}{r}\right)=N+\frac{k}{2}-\frac{2}{r}\left(\frac{N}{2}+\frac{|\gamma_2|}{2}\right)\ge  $$
$$ \ge  N+\frac{k}{2}-\frac{2}{r}\left(\frac{N}{2}+\frac{k-1}{2}\right). $$
So, choosing $r$ sufficiently large 
$$ \|I_2\|_{L^\infty(\mathbb R^N)}\lesssim t^{-N-\frac{k}{2}+\frac{2}{r}\left(\frac{N}{2}+\frac{k-1}{2}\right)}\|z\|_{X^s_k(\mathbb R^N)}\|w\|_{X^s_k(\mathbb R^N)}\int_{t/2}^t (t-\tau)^{-\frac{1}{2}-\frac{N}{2r}}d\tau\lesssim $$
$$ \lesssim t^{-N-\frac{k}{2}+\frac{1}{2}+\frac{1}{r}\left(\frac{N}{2}+k-1\right)}\|z\|_{X^s_k(\mathbb R^N)}\|w\|_{X^s_k(\mathbb R^N)}\lesssim t^{-\frac{N}{2}-\frac{k}{2}}\|z\|_{X^s_k(\mathbb R^N)}\|w\|_{X^s_k(\mathbb R^N)}\quad t\ge 2. $$
The estimate for $\nabla D^\beta$ in $L^\infty(\mathbb R^N)$ can be done as before.
\end{proof}
\begin{lem}\label{l.dec.3.RN}
    Let $k\in\mathbb N$ and $s>0$ such that $s-k> \frac{N}{2}-1$, let $\zeta\in X^s_k(\mathbb R^N)$ and $h,\theta_1,\theta_2,\theta_3\in X^{s+1}_{k+1}(\mathbb R^N)$, then it holds
    \begin{equation}
        \left\|\int_0^t e^{\Delta(t-\tau)}\zeta(\tau)\nabla h(\tau)d\tau\right\|_{X^{s+1}_{k+1}(\mathbb R^N)}\le C(N,s,k)\|\zeta\|_{X^{s}_k(\mathbb R^N)}\|h\|_{X^{s+1}_{k+1}(\mathbb R^N)},
    \end{equation}
     \begin{equation}
        \left\|\int_0^t e^{\Delta(t-\tau)}\nabla \theta_1(\tau)\nabla \theta_2(\tau)\theta_3(\tau)d\tau\right\|_{X^{s+1}_{k+1}(\mathbb R^N)}\le C(N,s,k)\|\theta_1\|_{X^{s+1}_{k+1}(\mathbb R^N)}\|\theta_2\|_{X^{s+1}_{k+1}(\mathbb R^N)}\|\theta_3\|_{X^{s+1}_{k+1}(\mathbb R^N)}.
    \end{equation}
\end{lem}
\begin{proof}\hfill\\
The function
$$ d(t)=\int_0^t e^{\Delta(t-\tau)}\left(\zeta(\tau)\nabla h (\tau) + \nabla \theta_1(\tau)\nabla \theta_2(\tau)\theta_3(\tau)\right)d\tau, $$
solves the system
$$ \left\{ \begin{array}{ll}
    (\partial_t-\Delta)d= \zeta\nabla h + \nabla \theta_1\nabla\theta_2\theta_3 & \mathbb R_+\times\mathbb R^N \\
    d(0)=0 & \mathbb R^N. 
\end{array}\right. $$
Thanks to Theorem \ref{t.lin.ex.RN} (considering only the heat equation) and using the estimates \eqref{bil.e.5}, \eqref{bil.e.8}, \eqref{tril.e.3} and \eqref{tril.e.4} we get that 
$$ \|d\|_{X^{s+1}(\mathbb R^N)} \lesssim \|\zeta\|_{X^s(\mathbb R^N)}\|h\|_{X^{s+1}(\mathbb R^N)} + \|\theta_1\|_{X^{s+1}(\mathbb R^N)}\|\theta_2\|_{X^{s+1}(\mathbb R^N)}\|\theta_3\|_{X^{s+1}(\mathbb R^N)}. $$
For what concerns the $L^\infty$ estimates, they can be done as in Lemma \ref{l.dec.2.RN}, using the fact that 
$$ \nabla h,\nabla \theta_j\in X^s_k(\mathbb R^N)\quad j=1,2,3,\quad \theta_3\in L^\infty(\mathbb R_+;L^\infty(\mathbb R^N)), $$
where the last fact follows by Sobolev embedding with $s>\frac{N}{2}-1$.
\end{proof}
The estimates for the half-space are similar, but unlike the $\mathbb R^N$ case, it is not always true that 
$$ \nabla e^{At}=e^{At}\nabla, $$
for $A=\mathbb P\Delta_D,\Delta_N$. For this reason, we prove decay only for $u$, for $d$ and $\nabla d$:
\begin{lem}\label{l.dec.1.hs}
    Let $s\in\left(\frac{1}{2},1\right]$, let $u_0\in H^s\cap L^1(\mathbb R^3_+)$ and $d_0\in H^{s+1}\cap L^1(\mathbb R^3_+)$, then
    $$ \left\|e^{\mathbb P\Delta_D t}u_0\right\|_{X^s_0(\mathbb R^3_+)}\le C(s)\|u_0\|_{H^s\cap L^1(\mathbb R^3_+)}, $$
    $$ \left\|e^{\Delta_Nt}d_0\right\|_{X^{s+1}_1(\mathbb R^3_+)}\le C(s)\|d_0\|_{H^{s+1}\cap L^1(\mathbb R^3_+)}. $$
\end{lem}
The proof can be done as in Lemma \ref{l.dec.1.RN}. Let us pass to the inhomogeneous term:
\begin{lem}\label{l.dec.2.hs}
    Let $s\in\left(\frac{1}{2},1\right]$, let $z,w,\zeta\in X^s_0(\mathbb R^3_+)$ and $h,\theta,\theta_1,\theta_2,\theta_3\in X^{s+1}_{1}(\mathbb R^3_+)$, then it holds
    \begin{equation}
        \left\|\int_0^t e^{\mathbb P\Delta_D(t-\tau)}\mathbb Pw(\tau)\nabla z(\tau)d\tau\right\|_{X^{s}_{0}(\mathbb R^3_+)}\le C(s)\|z\|_{X^{s}_0(\mathbb R^3_+)}\|w\|_{X^{s}_{0}(\mathbb R^3_+)},
    \end{equation}
    \begin{equation}
        \left\|\int_0^t e^{\mathbb P\Delta_D(t-\tau)}\mathbb P\nabla \theta(\tau)\nabla^2 h(\tau)d\tau\right\|_{X^{s}_{0}(\mathbb R^3_+)}\le C(s)\|\theta\|_{X^{s+1}_1(\mathbb R^3_+)}\|h\|_{X^{s+1}_{1}(\mathbb R^3_+)},
    \end{equation}
    \begin{equation}
        \left\|\int_0^t e^{\Delta_N(t-\tau)}\zeta(\tau)\nabla h(\tau)d\tau\right\|_{X^{s+1}_{1}(\mathbb R^3_+)}\le C(s)\|\zeta\|_{X^{s}_0(\mathbb R^3_+)}\|h\|_{X^{s+1}_{1}(\mathbb R^3_+)},
    \end{equation}
     \begin{equation}
        \left\|\int_0^t e^{\Delta_N(t-\tau)}\nabla \theta_1(\tau)\nabla \theta_2(\tau)\theta_3(\tau)d\tau\right\|_{X^{s+1}_{1}(\mathbb R^3_+)}\le C(s)\|\theta_1\|_{X^{s+1}_{1}(\mathbb R^3_+)}\|\theta_2\|_{X^{s+1}_{1}(\mathbb R^3_+)}\|\theta_3\|_{X^{s+1}_{1}(\mathbb R^3_+)}.
    \end{equation}
\end{lem}
\begin{proof}\hfill\\
Let 
$$ u= \int_0^t e^{\mathbb P\Delta_D(t-\tau)}\mathbb P\left(w(\tau)\nabla z(\tau) + \nabla \theta(\tau)\nabla^2 h(\tau)\right)d\tau, $$
$$ d= \int_0^t e^{\Delta_N(t-\tau)}\left(\zeta(\tau)\nabla h(\tau) + \nabla \theta_1(\tau)\nabla \theta_2(\tau)\theta_3(\tau)\right)d\tau, $$
then using Theorem \ref{t.lin.ex.dom.}, estimate \eqref{bil.e.5} and estimates from \eqref{bil.e.7} to \eqref{tril.e.4} we get as before that
$$ \|u\|_{X^s(\mathbb R^3_+)}\lesssim \|z\|_{X^s(\mathbb R^3_+)}\|w\|_{X^s(\mathbb R^3_+)} + \|\theta\|_{X^{s+1}(\mathbb R^3_+)}\|h\|_{X^{s+1}(\mathbb R^3_+)}, $$
$$ \|d\|_{X^{s+1}(\mathbb R^3_+)}\lesssim \|\zeta\|_{X^s(\mathbb R^3_+)}\|h\|_{X^{s+1}(\mathbb R^3_+)} + \|\theta_1\|_{X^{s+1}(\mathbb R^3_+)}\|\theta_2\|_{X^{s+1}(\mathbb R^3_+)}\|\theta_3\|_{X^{s+1}(\mathbb R^3_+)}. $$
The rest can be done as before: for the estimate in $X^s_0(\mathbb R^3_+)$ we can repeat the argument for Lemma \ref{l.dec.2.RN}. On the other hand,
$$ \partial_je^{\Delta_Nt}=\left\{\begin{array}{ll}
    e^{\Delta_Nt}\partial_j & j=1,2 \\
    e^{\Delta_Dt}\partial_3 & j=3.
\end{array}\right.$$
So we can repeat again the previous argument in order to get the estimate in $X^{s+1}_1(\mathbb R^3_+)$.
\end{proof}

\subsection{Decay Estimate in Exterior Domains}

For the exterior domains, as it can be expected, we have weaker results. As we saw in the previous section, we got the decay in time using properly the semigroup estimates. In exterior domains it is hard to treat derivatives greater than two for semigroups and therefore, the nonlinearity ${\rm Div}(\nabla d\odot \nabla d)$ is complicated to be dealt with. For this reason, we consider weaker spaces:

\begin{defn}\label{def.T.}
    Let $s\ge 0$, then we define
    $$ \Theta^s_u(\Omega)\coloneqq \left\{w\in X^s(\Omega)\mid \|w\|_{\Theta^s_u(\Omega)}<+\infty\right\}, $$
    $$ \Theta^s_d(\Omega)\coloneqq \left\{h\in X^{s+1}(\Omega)\mid \|h\|_{\Theta^s_d(\Omega)}<+\infty\right\}, $$
    where 
    $$ \|w\|_{\Theta^s_u(\Omega)}\coloneqq \|w\|_{X^s(\Omega)} + \sup_{t>0} t^{\frac{1}{2}}\|w(t)\|_{L^\infty(\Omega)}, $$
    $$ \|h\|_{\Theta^s_d(\Omega)}\coloneqq \|h\|_{X^{s+1}(\Omega)} + \sup_{t\le 2} t^{\frac{1}{2}}\|\nabla h(t)\|_{L^\infty(\Omega)} + \sup_{t\ge 1}t^{\frac{N}{2}+\frac{1}{2}}\|\nabla h(t)\|_{L^\infty(\Omega)}. $$
\end{defn}
As the name suggests, the space $\Theta^s_u(\Omega)$ is the one devoted for the velocity field $u$, while $\Theta^s_d(\Omega)$ is devoted to the direction field $d$. As before, we start our estimates from the homogeneous terms:
\begin{lem}\label{l.dec.ex.1}
    Let $N\ge 3$, $\Omega\subseteq\mathbb R^N$ exterior domain with sufficiently smooth boundary, let $s\in[0,1]$ then
    $$ \left\|e^{\mathbb P\Delta_Dt}u_0\right\|_{\Theta^s_u(\Omega)} + \left\|e^{\Delta_N t}d_0\right\|_{\Theta^s_d(\Omega)}\le C(\Omega,s)\left[\|u_0\|_{H^s(\Omega)} + \|d_0\|_{H^{s+1}\cap L^1\cap L^\infty(\Omega)}\right]. $$
\end{lem}
The proof can be done as in Lemma \ref{l.dec.1.RN} with the help of Theorem \ref{t.der.sem.}. Let us pass to the inhomogeneous terms:
\begin{lem}\label{l.dec.ex.2}
    Let $\Omega\subseteq\mathbb R^N$ exterior domain with sufficiently smooth boundary, let $s>\frac{N}{2}-1$ and let $z,w\in\Theta^s_u(\Omega)$ and $\theta,h\in \Theta^s_d(\Omega)$, then
    $$ \left\|\int_0^te^{\mathbb P\Delta_Dt}\mathbb Pw(\tau)\nabla z(\tau)d\tau\right\|_{\Theta^s_u(\Omega)}\le C(\Omega,s) \|z\|_{\Theta^s_u(\Omega)}\|w\|_{\Theta^s_u(\Omega)}, $$
    $$ \left\|\int_0^te^{\mathbb P\Delta_Dt}\mathbb P\nabla \theta(\tau)\nabla^2h(\tau)d\tau\right\|_{\Theta^s_u(\Omega)}\le C(\Omega,s) \|z\|_{\Theta^s_u(\Omega)}\|w\|_{\Theta^s_u(\Omega)}. $$
\end{lem}
\begin{proof}\hfill\\
    Since $\nabla h,\nabla \theta\in \Theta^s_u(\Omega)$ by definition of $\Theta^s_d(\Omega)$, we give the proof just for the first estimate: the estimate in the $X^s(\Omega)$-norm can be done as before by Theorem \ref{t.lin.ex.dom.}. Let us focus then on the estimate in $L^\infty$ in time and space: 
    $$ \left\|\int_0^t e^{\mathbb P\Delta_D t}\mathbb Pw(\tau)\nabla z(\tau)d\tau\right\|_{L^\infty(\Omega)}\lesssim \int_0^t (t-\tau)^{-\frac{1}{2}}\|w(\tau)\nabla z(\tau)\|_{L^N(\Omega)}d\tau\le $$
    $$ \le \|w\|_{\Theta^s_u(\Omega)}\int_0^t (t-\tau)^{-\frac{1}{2}}\tau^{-\frac{1}{2}}\|\nabla z(\tau)\|_{L^N(\Omega)}d\tau. $$
    Since $s>\frac{N}{2}-1$, we can find $r_0\in(s,s+1)$ such that $L^N(\Omega)\hookrightarrow H^{r_0}(\Omega)$. In particular
    $$ \|\nabla z(\tau)\|_{L^N(\Omega)}\lesssim \|z(\tau)\|_{H^{r_0}(\Omega)}. $$
    On the other hand, by interpolation for any $r\in(s,s+1)$ it holds
    $$ \|z(\tau)\|_{H^r(\Omega)}\lesssim \|z(\tau)\|_{H^s(\Omega)}^{s+1-r}\|z(\tau)\|_{H^{s+1}(\Omega)}^{r-s}. $$
    So, for any $\varepsilon\in(0,1)$ sufficiently small, we can take $r_\varepsilon\in[r_0,s+1)$ sufficiently near to $s+1$ such that
    $$ \|z(\tau)\|_{H^{r_0}(\Omega)}\le \|z(\tau)\|_{H^{r_\varepsilon}(\Omega)}\lesssim \|z(\tau)\|_{H^s(\Omega)}^{\varepsilon}\|z(\tau)\|_{H^{s+1}(\Omega)}^{1-\varepsilon}. $$
    Therefore, by Lemma \ref{l.tec.int.1}, for any $\varepsilon>0$ sufficiently small we get
    $$ \|w\|_{\Theta^s_u(\Omega)}\int_0^t (t-\tau)^{-\frac{1}{2}}\tau^{-\frac{1}{2}}\|\nabla z(\tau)\|_{L^N(\Omega)}d\tau\lesssim $$
    $$ \lesssim \|z\|_{\Theta^s_u(\Omega)}\|w\|_{\Theta^s_u(\Omega)} \left(\int_0^t (t-\tau)^{-\frac{1}{1+\varepsilon}}\tau^{-\frac{1}{1+\varepsilon}}d\tau\right)^\frac{1+\varepsilon}{2}\lesssim t^{-1+\frac{1+\varepsilon}{2}}\|z\|_{\Theta^s_u(\Omega)}\|w\|_{\Theta^s_u(\Omega)}. 
    $$
    Finally, since the estimate holds for any $\varepsilon\in(0,1)$ sufficiently small, we can take the limit as $\varepsilon\to0^+$ to conclude.
\end{proof}
\begin{lem}\label{l.dec.ex.3}
    Let $\Omega\subseteq\mathbb R^N$ exterior domain with sufficiently smooth boundary, let $s>\frac{N}{2}-1$ and let $z\in\Theta^s_u(\Omega)$ and $h,\theta_1,\theta_2,\theta_3\in \Theta^s_d(\Omega)$, then
    $$ \left\|\int_0^te^{\Delta_Nt}z(\tau)\nabla h(\tau)d\tau\right\|_{\Theta^s_d(\Omega)}\le C(\Omega,s)\|z\|_{\Theta^s_u(\Omega)}\|h\|_{\Theta^s_d(\Omega)}, $$
    $$ \left\|\int_0^te^{\Delta_Nt}\nabla \theta_1(\tau)\nabla \theta_2(\tau)\theta_3(\tau)d\tau\right\|_{\Theta^s_d(\Omega)}\le C(\Omega,s)  \|\theta_1\|_{\Theta^s_d(\Omega)}\|\theta_2\|_{\Theta^s_d(\Omega)}\|\theta_3\|_{\Theta^s_d(\Omega)}. $$
\end{lem}
\begin{proof}\hfill\\
The estimate in the $X^{s+1}(\Omega)$-norm can be done as before by Theorem \ref{t.lin.ex.dom.}. For the $L^\infty$ estimates, the boundedness in $t\le 2$ can be done as in Lemma \ref{l.dec.ex.2}. Let us prove now the decay for $t\ge2$ (the case $t\in[1,2]$ follows from the estimate near $t=0$). We consider just the first estimate, the second one can be done similarly using $\theta_3\in L^\infty(\mathbb R_+;L^\infty(\Omega))$. We split the integral in two pieces
$$ \int_0^t e^{\Delta_N(t-\tau)}z(\tau)\nabla h(\tau) d\tau= I_1 + I_2,  $$
where 
$$ I_1=\int_0^{t/2} e^{\Delta_N(t-\tau)}z(\tau)\nabla h(\tau) d\tau, $$
$$ I_2= \int_{t/2}^t e^{\Delta_N(t-\tau)}z(\tau)\nabla h(\tau) d\tau. $$
Then
$$ \|I_1\|_{L^\infty(\Omega)}\lesssim \int_0^{t/2}(t-\tau)^{-\frac{N}{2}}\|z(\tau)\nabla h(\tau)\|_{L^1(\Omega)}d\tau\lesssim $$
$$ \lesssim t^{-\frac{N}{2}}\int_0^{t/2} \|z(\tau)\|_{L^2(\Omega)}\|\nabla h(\tau)\|_{L^2(\Omega)} d\tau \le t^{-\frac{N}{2}}\|z\|_{\Theta^s_u(\Omega)}\|h\|_{\Theta^s_d(\Omega)}. $$
On the other hand
$$ \|I_2\|_{L^\infty(\Omega)} \lesssim \int_{t/2}^t (t-\tau)^{-\frac{1}{2}}\|w(\tau)\nabla z(\tau)\|_{L^N(\Omega)}d\tau. $$
We notice that
$$ \|z(\tau)\nabla h(\tau)\|_{L^N(\Omega)}\le \tau^{-\frac{N}{2}-\frac{1}{2}}\|h\|_{\Theta^s_d(\Omega)}\|z(\tau)\|_{L^N(\Omega)}\lesssim \tau^{-\frac{N}{2}-\frac{1}{2}}\|h\|_{\Theta^s_d(\Omega)}\|z\|_{\Theta^s_u(\Omega)}, $$
where we used that $L^N(\Omega)\hookrightarrow H^s(\Omega)$ for $s>\frac{N}{2}-1$. Therefore
$$ \|I_2\|_{L^\infty(\mathbb R^N)}\lesssim  t^{-\frac{N}{2}-\frac{1}{2}}\|z\|_{\Theta^s_u(\Omega)}\|h\|_{\Theta^s_d(\Omega)}\int_{t/2}^t (t-\tau)^{-\frac{1}{2}}d\tau \le  $$
$$ \le t^{-\frac{N}{2}}\|z\|_{\Theta^s_u(\Omega)}\|h\|_{\Theta^s_d(\Omega)}. $$
The estimate for the gradient is similar: 
$$ \|\nabla_xI_1\|_{L^\infty(\Omega)}\lesssim \int_0^{t/2}(t-\tau)^{-\frac{N}{2}-\frac{1}{2}}\|z(\tau)\nabla h(\tau)\|_{L^1(\Omega)}d\tau\lesssim $$
$$ \lesssim t^{-\frac{N}{2}-\frac{1}{2}}\int_0^{t/2} \|z(\tau)\|_{L^2(\Omega)}\|\nabla h(\tau)\|_{L^2(\Omega)} d\tau \le t^{-\frac{N}{2}-\frac{1}{2}}\|z\|_{\Theta^s_u(\Omega)}\|h\|_{\Theta^s_d(\Omega)}. $$
On the other hand
$$ \|\nabla_xI_2\|_{L^\infty(\Omega)} \lesssim \int_{t/2}^t (t-\tau)^{-\frac{1}{2}}\|z(\tau)\nabla h(\tau)\|_{L^\infty(\Omega)}d\tau\le  $$
$$ \le \|z\|_{\Theta^s_u(\Omega)}\|h\|_{\Theta^s_d(\Omega)}\int_{t/2}^t (t-\tau)^{-\frac{1}{2}}\tau^{-\frac{N}{2}-1}d\tau\lesssim t^{-\frac{N}{2}-\frac{1}{2}}\|z\|_{\Theta^s_u(\Omega)}\|h\|_{\Theta^s_d(\Omega)}. $$
\end{proof}

\subsection{Proofs of Theorems \ref{t.gl.ex.RN} and \ref{t.gl.ex.dom.}}

As for the local case, we need a Lemma which ensures that $|d_0+\eta|=1$ implies $|d(t)+\eta|=1$ for a.e. $t>0$:
\begin{lem}\label{l.un.gl.}
    Let $s> \frac{N}{2}-1$, $\eta\in \mathbb R^N$,
    $$ u_0\in H^s_{\mathbb P\Delta_D}\cap L^1\left(\Omega;\mathbb{R}^N\right),\quad d_0\in H^{s+1}_{\Delta_N}\cap L^1\left(\Omega;\mathbb R^N\right), $$
    with $|\eta+d_0|=1$, let $(u,d)\in X^s(\Omega)\times X^{s+1}(\Omega)$ be a solution for \eqref{sys.gl-red.}, then $|\eta+d(t)|=1$ for a.e. $t>0$.
\end{lem}
The proof is the same of the local case, with the help of estimates from \eqref{bil.e.5} to \eqref{tril.e.4} and Lemmas \ref{l.dec.2.RN} to \ref{l.dec.2.hs}.

We are now ready to prove the global existence and decay results. We give the details of Theorem \ref{t.gl.ex.RN}, the proof of Theorem \ref{t.gl.ex.dom.} is similar.
\begin{proof}[Proof of Theorem \ref{t.gl.ex.RN}]
We define the space
$$ Y\coloneqq \{(w,\theta)\in X^s_k(\mathbb R^N)\times X^{s+1}_{k+1}(\mathbb R^N)\mid w(t)\in J_2(\mathbb R^N)\:\:\text{for a.e.}\:\:t>0, \:\:\|(w,\theta)\|_{Y}<+\infty\},  $$
where $k\in\mathbb N$ is such that $s-k>\frac{N}{2}-1$ and
$$ \|(w,\theta)\|_{Y}\coloneqq\|w\|_{X^s_k(\mathbb R^N)}+\|\theta\|_{X^{s+1}_{k+1}(\mathbb R^N)}. $$
Now we define the map $\Phi\colon Y\to Y$ such that the function $\Phi(w,\theta)=(u,d)$ solves
\begin{equation}
    \left\{\begin{array}{ll}
    (\partial_t-\mathbb{P}\Delta)u=-\mathbb{P}(w\cdot \nabla w +{\rm Div}(\nabla \theta\odot\nabla \theta)) & \mathbb R_+\times\mathbb R^N \\
    {\rm div}u=0 & \mathbb R_+\times \mathbb R^N \\
    (\partial_t-\Delta)d=-w\cdot \nabla \theta+|\nabla \theta|^2(\eta+\theta) & \mathbb R_+\times\mathbb R^N \\
    u(0)=u_0,\quad d(0)=d_0 & \mathbb R^N,
    \end{array}\right.
\end{equation}
where $d_0\coloneqq v_0-\eta$. Again, we want to prove that $\Phi$ is a contraction argument in order to prove the existence of a fixed point. 

\textbf{Step 1}: From estimates \eqref{bil.e.5} to \eqref{tril.e.4} and from Lemma \ref{l.dec.1.RN} to Lemma \ref{l.dec.3.RN} we get
\begin{equation}\label{nl.gl.1}
\begin{aligned}
   & \|\Phi(w,\theta)\|_{Y}\lesssim \\
 \lesssim  \|u_0\|_{H^s\cap L^1(\mathbb R^N)}+ & \|d_0\|_{H^{s+1}\cap L^1(\mathbb R^N)} + \|(w,\theta)\|_{Y}^2+\|(w,\theta)\|_{Y}^3,
\end{aligned}
\end{equation}
\begin{equation}\label{nl.gl.2}
\begin{aligned}
    & \|\Phi(w_1,\theta_1)-\Phi(w_2,\theta_2)\|_{Y}\lesssim \\
    \lesssim & \left(\|(w_1,\theta_1)\|_Y + \|(w_2,\theta_2)\|_Y + \|(w_1,\theta_1)\|_{Y}^2+\|(w_2,\theta_2)\|_{Y}^2\right)\|(w_1,\theta_1)-(w_2,\theta_2)\|_{Y}.
\end{aligned}
\end{equation}

\textbf{Step 2}: Let $\varepsilon>0$, so we define the space
$$ Z_\varepsilon\coloneqq \{(w,\theta)\in Y\mid w(0)=u_0,\quad \theta(0)=d_0 ,\quad \|(w,\theta)\|_{Y}\le 2C\varepsilon\}, $$
where $C>0$ comes from the estimate \eqref{nl.gl.1}: 
$$ \|\Phi(w,\theta)\|_{Y}\le C\left[\|u_0\|_{H^s\cap L^1(\mathbb R^N)}+\|d_0\|_{H^{s+1}\cap L^1(\mathbb R^N)} + \varepsilon^2+\varepsilon^3\right]\le $$
$$ \le C\varepsilon\left[1 + \varepsilon^2 + \varepsilon^3\right]. $$
We notice that, if we take $\varepsilon>0$ such that 
$$ 1+\varepsilon^2 + \varepsilon^3\le 2, $$
then $\Phi\colon Z_\varepsilon\to Z_\varepsilon$. Moreover, $\Phi$ is a contraction if we choose $\varepsilon>0$ sufficiently small: thanks to \eqref{nl.gl.2} there is $M>0$ such that, for any $(w_1,\theta_1),(w_2,\theta_2)\in Z_\omega$, it holds 
$$ \|\Phi(w_1,\theta_1)-\Phi(w_2,\theta_2)\|_{Y}\le 2MC\varepsilon(1+2C\varepsilon) \|(w_1,\theta_1)-(w_2,\theta_2)\|_{Y}, $$
so we can take $\varepsilon\le [2MC(1+2C\varepsilon)]^{-1}$ in order to have a contraction. Therefore, we get a solution for the system \eqref{EL.sys-proj.}.

\vspace{2mm}

\textbf{Step 3:} The solution we found is unique in $X^s(\mathbb R^N)\times X^{s+1}(\mathbb R^N)$: let $(u_1,d_1), (u_2,d_2)\in X^s(\mathbb R^N)\times X^{s+1}(\mathbb R^N)$ and let 
$$ R_1=\|(u_1,d_1)\|_{X^s(\mathbb R^N)\times X^{s+1}(\mathbb R^N)},\quad R_2=\|(u_2,d_2)\|_{X^s(\mathbb R^N)\times X^{s+1}(\mathbb R^N)}. $$
As we did in the local case, similarly to the proof of the estimate \eqref{nl.gl.2}, we can find $T_0=T_0(R_1,R_2)$ such that 
$$ u_1(t,x)=u_2(t,x),\quad d_1(t,x)=d_2(t,x)\quad\forall t\in(0,T_0)\:\: \text{for a.e.}\:\:x\in \mathbb R^N. $$
Then we conclude: let us suppose by contradiction it exists $t_0\in\mathbb R_+$ such that the couples of functions are not equal in a set of measure strictly positive, then by the previous argument, with a finite number of steps, we get the contradiction. 

\vspace{2mm}

\textbf{Step 4:} Finally, we consider $v(t,x)=d(t,x)+\eta$. Thanks to Lemma \ref{l.un.gl.}, the couple $(u,v)$ solves the Ericksen-Leslie system \eqref{EL.sys-proj.}. The existence of $p$ follows from Theorem \ref{t.lin.ex.RN}.
\end{proof}

\bibliographystyle{plain}
 \bibliography{EL_BG}

\end{document}